\documentclass[11pt]{article}
\usepackage[left=1in,top=1in,right=1in,bottom=1in,head=.1in,nofoot]{geometry}

\usepackage{setspace,url,bm,amsmath} 

\usepackage[numbers]{natbib}
\usepackage{graphicx} 
\usepackage{bbm}
\usepackage{latexsym}
\usepackage{caption}
\usepackage[margin=20pt]{subcaption}
\usepackage{hyperref}
\usepackage[]{algorithm2e}
\usepackage{booktabs}

\usepackage[table]{xcolor}

\newcommand{\GG}[1]{}

\usepackage{amsthm}
\usepackage{amssymb}
\usepackage{amsmath}
\usepackage{color}

\usepackage{booktabs}

\usepackage{comment}
\theoremstyle{definition}

\newtheorem*{theorem*}{Theorem}
\newtheorem*{rmk*}{Remark}

\newtheorem{theorem}{Theorem}
\newtheorem{proposition}{Proposition}
\newtheorem{lemma}{Lemma}
\newtheorem{example}{Example}

\newtheorem{corollary}{Corollary}
\newtheorem*{corollary*}{Corollary}

\usepackage{mathrsfs}

\usepackage{natbib} 
\bibpunct{(}{)}{;}{a}{}{,} 

\usepackage{color}
\usepackage{listings}

\def\Var{\text{Var}}
\def\Cov{\text{Cov}}

\def\converged{\stackrel{d}{\longrightarrow}}

\def\converged{\stackrel{d}{\longrightarrow}}
\def\ybars{\bar{y}_{\text{S}}}
\def\ybarsq{\bar{y}_{\text{S}q}}

\DeclareFontFamily{U}{mathx}{\hyphenchar\font45}
\DeclareFontShape{U}{mathx}{m}{n}{
      <5> <6> <7> <8> <9> <10>
      <10.95> <12> <14.4> <17.28> <20.74> <24.88>
      mathx10
      }{}
\DeclareSymbolFont{mathx}{U}{mathx}{m}{n}
\DeclareFontSubstitution{U}{mathx}{m}{n}
\DeclareMathAccent{\widecheck}{0}{mathx}{"71}
\DeclareMathAccent{\wideparen}{0}{mathx}{"75}

\usepackage[normalem]{ulem}
\usepackage{xcolor}
\newcommand\redsout{\bgroup\markoverwith{\textcolor{red}{\rule[0.5ex]{2pt}{0.4pt}}}\ULon}

\allowdisplaybreaks

\begin{document}
\doublespacing
\title{\bf General forms of finite population central limit theorems with applications to causal inference}

\author{Xinran Li and Peng Ding
\footnote{
Xinran Li is Doctoral Candidate, Department of Statistics, Harvard University,  Cambridge, MA 02138 (E-mail: xinranli@fas.harvard.edu). Peng Ding is Assistant Professor, Department of Statistics, University of California, Berkeley, CA 94720 (E-mail: pengdingpku@berkeley.edu).
The authors thank Mr. Yotam Shem-Tov and Ms. Lo-Hua Yuan for comments, and Professor Peter J. Bickel for the references of finite population inference. 
}
}
\date{}
\maketitle

\begin{abstract}
Frequentists' inference often delivers point estimators associated with confidence intervals or sets for parameters of interest. Constructing the confidence intervals or sets requires understanding the sampling distributions of the point estimators, which, in many but not all cases, are related to asymptotic Normal distributions ensured by central limit theorems. Although previous literature has established various forms of central limit theorems for statistical inference in super population models, we still need general and convenient forms of central limit theorems for some randomization-based causal analysis of experimental data, where the parameters  of interests are functions of a finite population and randomness comes solely from the treatment assignment. We use central limit theorems for sample surveys and rank statistics to establish general forms of the finite population central limit theorems that are particularly useful for proving asymptotic distributions of randomization tests under the sharp null hypothesis of zero individual causal effects, and for obtaining the asymptotic repeated sampling distributions of the causal effect estimators. The new central limit theorems hold for general experimental designs with multiple treatment levels and multiple treatment factors, and are immediately applicable for studying the asymptotic properties of many methods in causal inference, including instrumental variable, regression adjustment, rerandomization, clustered randomized experiments, and so on. Previously, the asymptotic properties of these problems are often based on heuristic arguments, which in fact rely on general forms of finite population central limit theorems that have not been established before.  Our new theorems fill in this gap by providing more solid theoretical foundation for asymptotic randomization-based causal inference.


\end{abstract}

\section{Introduction}

Central limit theorems (CLTs) are central pillars of many frequentists' inferential procedures. Most CLTs assume that the observations are samples from a hypothetical infinite super population model \citep[e.g.][]{lehmann1999elements, van2000asymptotic}. In  sample surveys and randomized experiments, however, the infinite super population seems contrived, and the parameters of interests are functions of the attributes of well-defined finite units. Finite population inference requires no assumptions on the data generating process of the units, and quantifies the uncertainty based on the randomness from the study design. In sample surveys, the population has fixed quantities of interest, and the sampling process induces randomness in the estimators \citep[cf.][]{cochran1977}; in randomized experiments, the potential outcomes \citep{Neyman:1923, Rubin:1974} of the experimental units are fixed, and the physical randomization acts as the ``reasoned basis'' \citep{Fisher:1935} for conducting statistical testing and estimation \citep[cf.][]{kempthorne1952design, kempthorne1994design, rosenbaum2002observational, abadie2014finite, imbens2015causal}. This is sometimes called randomization-based or design-based inference, dating back to the classical analysis of sample surveys \citep{splawa1925contributions, neyman1934two} and randomized experiments \citep{Neyman:1923, neyman1935statistical, Fisher:1935}.

For simple random sampling, \citet{erdos1959central}, \citet{hajek1960limiting} and \citet{madow1948limiting} obtained various forms of CLTs, with a convenient form presented in \citet[][Appendix 4, Theorem 6]{lehmann2006nonparametrics} and \citet[][Theorem 2.8.2]{lehmann1999elements}. In fact, these theorems are special cases of the CLTs for rank statistics \citep{wald1944statistical, noether1949theorem, fraser1956vector, hajek1961some}. In randomized experiments, because the treatment and control groups are simple random samples from the finite experimental units, the CLTs for sampling surveys are sometimes adequate for establishing asymptotic distributions of the causal effect estimators \citep[e.g.][]{liu2014large, Ding:2014, Ding:2015}. Unfortunately, however, these CLTs do not immediately apply to estimators beyond the difference-in-means in treatment-control experiments. For instance, \citet{freedman2008regression_a, freedman2008regression_b} provided only an informal proof for the asymptotic Normality of the regression estimator in randomized experiments based on \citet{hoeffding1951combinatorial} and \citet{hoglund1978sampling}. Many other randomization-based causal inferences invoked CLTs implicitly without a formal proof, e.g., rerandomization in \citet{morgan2012rerandomization}, factorial experiments in \citet{dasgupta2014causal} and \citet{Ding:2014}, and clustered randomized experiments in \citet{Middleton2015cluster}.

Therefore, causal inference needs general forms of CLTs that apply to more than two treatment levels, more complex designs than  completely randomized experiments, and more complex estimators than difference-in-means. We first recall a deep connection between sample surveys and randomized experiments \citep[cf.][]{Neyman:1923, splawa1925contributions, neyman1934two, neyman1935statistical, rubin1990comment, fienberg1996reconsidering}, and then utilize a CLT for rank statistics \citep{fraser1956vector} to establish the CLTs that are particularly useful for causal analysis of randomized experiments. The salient feature of the new CLTs is that the asymptotic variances and covariances depend on the correlation structure among the potential outcomes under different treatment levels. This feature did not appear in any CLTs for sample surveys and rank statistics, but did appear in the variance formula of the difference-in-means estimator in \citet{Neyman:1923}. Because of the generality of the new CLTs, they are readily applicable to many existing causal inference problems, including instrumental variable estimation, randomization tests with more than two treatment levels,
multiple randomization tests, rerandomization to ensure covariate balance \citep*{morgan2012rerandomization, morgan2015rerandomization, asymrerand2106},   regression adjustment for completely randomized experiments \citep{freedman2008regression_a, freedman2008regression_b,lin2013}, clustered randomized experiments \citep{Middleton2015cluster}, and unbalanced factorial experiments \citep{dasgupta2014causal}, etc. The new CLTs not only justify the asymptotic properties of some existing procedures, but also help to establish new results that did not appear in the previous literature. They will become useful tools for studying asymptotic properties of many randomization-based inferential procedures in causal inference.

Under the sharp null hypothesis with zero or general known unit-level causal effects, all the potential outcomes are known, and the randomization distribution of any test statistics can be computed exactly or at least simulated by Monte Carlo. In this case, the role of the CLTs is to give convenient approximations of the null distributions and provide statistical insights with explicit formulas. More importantly, without imposing the sharp null hypothesis as in the repeated sampling evaluations \citep{Neyman:1923}, the randomization distributions of the causal effect estimators depend on unknown values of the potential outcomes. In this case, the role of the CLTs is then not only to give convenient approximations but also allow for asymptotic statistical inference without knowing all the values of the potential outcomes. As shown in \citet{Neyman:1923}, this type of inference is often statistically conservative even asymptotically, which will be clearer with our general finite population CLTs in Section \ref{sec:fclt_cre}.

Below we first review some classical finite population CLTs for simple random sampling \citep{hajek1960limiting} and for random partition \citep[cf.][Appendix 8, Theorem 19]{lehmann2006nonparametrics}, and then establish new finite population CLTs that apply to general randomized experiments with multiple treatment levels. Throughout the paper, we use important causal inference problems to illustrate the CLTs. 
All technical details are in the Supplementary Material.

%

\section{Classical finite population CLTs}
\label{sec::fpclts-classical}

\subsection{Simple random sampling}\label{sec:SRS}
Consider a finite population $\Pi_N = \{ y_{N1}, y_{N2}, \ldots, y_{NN} \}$ with $N$ units.
The population mean, $\bar{y}_{N} = (y_{N1} + \cdots + y_{NN})/N$, is often of interest. A sample is a subset of $\Pi_N$ represented by the vector of inclusion indicators $(Z_1, \ldots, Z_N) \in \{0,1\}^N$, where $Z_i=1$ if the sample contains unit $i$ and $Z_i=0$ otherwise.
In simple random sampling,  
the probability that the inclusion indicator vector of $(Z_1, \ldots, Z_N)$ takes a particular value $(z_1,\ldots,z_N)$ is $n!(N-n)!/N!$, where $\sum_{i=1}^N z_i = n$ and $\sum_{i=1}^N (1-z_i) = N-n.$
The sample average $\ybars = \sum_{i=1}^N Z_i y_{Ni}/n$ is an intuitive estimator for the population mean. 
In the formula of $\ybars$, randomness comes from $(Z_1, \ldots, Z_N)$, and all the $y_{Ni}$'s are fixed population quantities. Because of this feature, it is straightforward to show that
$\ybars$ has mean $\bar{y}_N$ and variance
\begin{align}\label{eq::mean-var}
\Var(\ybars) =  \left( \frac{1}{n} - \frac{1}{N}\right) v_N,
\end{align}
depending on the finite population variance of $\Pi_N$ \citep[cf.][]{cochran1977}:
\begin{align}\label{eq:v_N}
v_N = \frac{1}{N-1}\sum_{i=1}^N\left(
y_{Ni}-\bar{y}_N
\right)^2.
\end{align}
To conduct statistical inference of $\bar{y}_N$ based on $\ybars$, we need to characterize the sampling distribution of $\ybars$ induced by simple random sampling. Although the exact distribution of $\ybars$ is complex, some asymptotic techniques help to use its first two moments to describe its asymptotic distribution. The finite population asymptotic scheme embeds $\Pi_N$ into a hypothetical infinite sequence of finite populations with increasing sizes, and the asymptotic distribution of any sample quantity is its limiting distribution along this hypothetical infinite sequence \citep[cf.][]{lehmann1999elements,lehmann2006nonparametrics,Ding:2014,aronow2014,Middleton2015cluster}. Similar to the classical Lindeberg--Feller CLT \citep{durrett2010probability}, the asymptotic behavior of $\ybars$ depends crucially on the maximum squared distance of the $y_{Ni}$'s from the population mean $\bar{y}_N$:
\begin{align}\label{eq:m_N}
m_N = \max_{1\leq i \leq N}(y_{Ni} - \bar{y}_N)^2 . 
\end{align}
The following theorem due to \citet{hajek1960limiting} states that, under some regularity conditions on the sequence of finite populations $\{ \Pi_N\}_{N=1}^{\infty}$ and sizes of simple random samples, the sample average is asymptotically Normal.

\begin{theorem}\label{thm:CLThajek}
Let $\bar{y}_{\text{S}}$ be the average of a simple random sample of size $n$ from a finite population $\Pi_N = \{ y_{N1}, y_{N2}, \ldots, y_{NN} \}$.
As $N\rightarrow \infty$, if 
\begin{align}
\label{eq:condi_CLThajek}
 \frac{1}{\min(n, N-n)}\cdot\frac{  m_N }{ v_N } \rightarrow 0,
\end{align} 
then $(\ybars - \bar{y}_N)/  \sqrt{ \Var(\ybars)  }   \converged \mathcal{N}(0,1)$, recalling that $v_N$ and $m_N$ are defined in \eqref{eq:v_N} and \eqref{eq:m_N}. 
\end{theorem}

\citet[][Appendix 4]{lehmann2006nonparametrics} presented a special case of a theorem in \citet{hajek1961some}, requiring an equivalent form of Condition \eqref{eq:condi_CLThajek} and additionally  $n\rightarrow \infty$ and $N-n \rightarrow \infty$. For the ease of  notation and  interpretation, we present Condition \eqref{eq:condi_CLThajek} in the main text, and give more technical details in the Supplementary Material. 
In fact, we can show that
$m_N/v_N\geq 1-N^{-1}$, and therefore
Condition \eqref{eq:condi_CLThajek} implies $n\rightarrow \infty$ and $N-n \rightarrow \infty$. Because a weighted sum of the means of the sampled units and unsampled units is fixed at the population mean, their asymptotic behaviors must be exactly the same up to some scaling factors. This further explains the symmetry of $n$ and $N-n$ in Condition \eqref{eq:condi_CLThajek}.

Condition \eqref{eq:condi_CLThajek} needs further explanations. 
Importantly, it does not depend on the scale of the $y_{Ni}$'s of the finite population. We can simply standardize the $y_{Ni}$'s to ensure a constant finite population variance (e.g., $v_N = 1$ for all $N$), and the values of $m_N$ for these finite populations change correspondingly.
We further assume that the proportion of sampled units has a limiting value $n/N\rightarrow \gamma_{\infty}\in [0,1]$. 
The value of $\gamma_{\infty}$ is usually positive in randomized experiments; we assign a proportion of units to the treatment group and a comparable proportion of units to the control group, and both groups are simple random samples from the finite units. The value of $\gamma_{\infty}$ may be zero in survey sampling when the sample fraction is extremely small. When $0<\gamma_\infty < 1$, Condition \eqref{eq:condi_CLThajek} is equivalent to ${m_N}/{N} \rightarrow 0$; when $\gamma_{\infty}=0$, it is equivalent to
$
{m_N}/{n} \rightarrow 0;
$
when $\gamma_{\infty}=1$, it is equivalent to $m_N/(N-n)\rightarrow\infty$.
It is apparent that when all units of $\Pi_N$ have bounded values, all equivalent forms must hold if both $n$ and $N-n$ go to infinity.
Moreover, if $0<\gamma_{\infty}<1$ and the units in $\Pi_N$ are independent and identically distributed (i.i.d) draws from a super population with more than two moments and a nonzero variance, then $v_N$ converges to the variance of the super population $v_{\infty}$, and Condition \eqref{eq:condi_CLThajek} and its
equivalent form ${m_N}/{N} \rightarrow 0$ hold with probability one, as commented in the Supplementary Material.
In this case, 
the asymptotic Normality reduces to
$
\sqrt{n}(\ybars-\bar{y}_N) \converged \mathcal{N}\left( 0, (1-\gamma_{\infty})v_{\infty}\right).
$

Separately in the literature, \citet{erdos1959central} and \citet{hajek1960limiting} established finite population CLTs for simple random sampling, and
\citet{wald1944statistical}, \citet{noether1949theorem}, \citet{hoeffding1951combinatorial}, \citet{motoo1956hoeffding}, and \citet{schneller1988short} established various forms of CLTs for rank statistics. \citet{madow1948limiting} used a CLT for rank statistics to prove a version of finite population CLT for simple random sampling. \citet{hajek1960limiting, hajek1961some}  and \citet{Robinson1972} 
discussed different forms of sufficient and necessary conditions. 


Theorem \ref{thm:CLThajek} suggests a strategy 
to construct a large-sample confidence interval for $\bar{y}_N$ based on the Normal approximation. This strategy requires us to consistently estimate the variance of $\ybars$.
The sample variance of the simple random sample
\begin{align*}
\widehat{v}_N = \frac{1}{n-1}\sum_{i:Z_i=1}
\left(
y_{Ni} - \ybars 
\right)^2
\end{align*}
is unbiased for the population variance $v_N$, and therefore
\begin{align*}
\widehat{\Var}(\ybars) = \left( \frac{1}{n}-\frac{1}{N}\right)\widehat{v}_N
\end{align*}
is unbiased for the variance of $\ybars$. 
\begin{proposition}\label{prop:v_N_consist}
Under the conditions in Theorem \ref{thm:CLThajek}, 
$\widehat{\Var}(\ybars)/\Var(\ybars)= \widehat{v}_N/v_N \overset{p}{\longrightarrow} 1$ as $N\rightarrow \infty.$
\end{proposition}
Therefore, Theorem \ref{thm:CLThajek} and Proposition \ref{prop:v_N_consist} justify the usual confidence interval for $\bar{y}_N$ based on the Normal approximation. This confidence interval is standard in the survey sampling literature \citep[eg.][]{cochran1977}, but to the best of our knowledge, the proof of the simple fact $\widehat{v}_N/v_N \overset{p}{\longrightarrow} 1$ was neglected
or derived under unnecessarily strong conditions 
 in the literature \cite[e.g.,][page 259]{lehmann1999elements}. In the randomization-based causal inference, Proposition \ref{prop:v_N_consist} is crucial for consistency of the variance estimators, but previous literature provided only heuristic arguments without formal proofs \citep[e.g.,][]{liu2014large, Ding:2014}.

Theorem \ref{thm:CLThajek} and Proposition \ref{prop:v_N_consist} have numerous applications. We review two examples below. For more applications in nonparametric tests and randomization-based causal inference, please see \citet{lehmann2006nonparametrics}, \citet{liu2014large}, \citet{Ding:2014}, and \citet{Ding:2015}.


\begin{example}[Normal approximation of the Hypergeometric distribution]\label{eg:hyper}
If all the units in the finite population $\Pi_N$ take binary values with $N_1$ of them being $1$, then 
$n\ybars$, the number of $1$'s in a simple random sample of size $n$, follows a Hypergeometric distribution. We verify in the Supplementary Material that, if $\Var(n\ybars)\rightarrow \infty$ then Condition \eqref{eq:condi_CLThajek} holds.
Therefore, $n\ybars$ is asymptotically Normal if its variance goes to infinity \citep{lehmann2006nonparametrics, Vatutin1982}. 
In fact, this is a sufficient and necessary condition \citep{kou1996asymptotics}.
Both Fisher's exact test and the randomization test for a binary outcome have null distributions depending on a Hypergeometric random variable \citep[cf.][]{Ding:2015}, and therefore can be efficiently computed using Normal approximations with large samples.
$\hfill \square$
\end{example}

\begin{example}
[Randomization-based instrumental variable estimation]\label{eg:instru_estimation}
Consider a completely randomized experiment with $N$ units, in which $n_1$ assigned to the treatment and $n_0$ assigned to the control. For unit $i$, let $Z_i$ be the binary indicator for treatment assignment, $D_i$ be the binary or continuous received dose of the treatment, and $Y_i$ be the response. Because both the dose and response are affected by the treatment, we define $(D_i(1), D_i(0))$ as the potential outcomes for the dose, and $(Y_i(1), Y_i(0))$ as the potential outcomes for the response. Under the linear instrumental variable model, $Y_i(1) - Y_i(0) = \beta \{  D_i(1) - D_i(0) \}$ for all unit $i$, where the coefficient $\beta$ is a measure of the dose-response relationship \citep{rosenbaum2002observational, imbens2005robust}. 
The model automatically satisfies the so-called exclusion restriction assumption, because $D_i(1) =D_i(0)$ implies $Y_i(1) = Y_i(0).$ Define the adjusted outcome as $A_i \equiv Y_i - \beta D_i$ with potential outcomes $A_i(z) = Y_i(z) - \beta D_i(z)$ under treatment $z.$
Because $Y_i(1) - \beta D_i(1) = Y_i(0) - \beta D_i(0)$, the adjusted outcome satisfies $A_i = A_i (1) = A_i(0)$, i.e., the treatment does not affect the observed value of $A_i $ at the true value of $\beta$. Therefore, we can construct a confidence interval for $\beta$ by inverting randomization tests. 
Although general test statistics such as rank statistics are useful in practice \citep{rosenbaum2002observational, imbens2005robust}, we use the difference-in-means of $A$ as the test statistic for simplicity. 
Let $(\widehat{\tau}_A, \widehat{\tau}_Y, \widehat{\tau}_D)$ be the difference-in-means between treatment and control groups, $(\bar{A}, \bar{Y}, \bar{D})$ the pooled means, and $(s_A^2, s_Y^2, s_D^2)$ the finite population variances of the pooled observed values of $A,Y$ and $D$. Let $s_{YD}$ be the finite population covariance between the pooled observed values of $Y$ and $D$. 
Then the test statistic has the following equivalent forms:
\begin{align}
\widehat{\tau}_A =  \frac{1}{n_1} \sum_{i=1}^N Z_i (Y_i - \beta D_i) - \frac{1}{n_0}\sum_{i=1}^N (1-Z_i) (Y_i - \beta D_i)
=
\widehat{\tau}_Y - \beta \widehat{\tau}_D
 =  \frac{N}{n_0}\left(
\frac{1}{n_1}\sum_{i=1}^N Z_i A_i - \bar{A}
\right).
\label{eq::adjusted-outcome}
\end{align}
Under the null hypothesis that $\beta$ is the true value, if $\{A_i:i=1,\ldots,N\}$ satisfy Condition \eqref{eq:condi_CLThajek} in Theorem \ref{thm:CLThajek}, 
 then according to the last equivalent form of the test statistic $\widehat{\tau}_A$ in \eqref{eq::adjusted-outcome}, it converges to a Normal distribution with mean $0$ and variance 
\begin{align*}
\Var_0(\widehat{\tau}_A) = 
\frac{N^2}{n_0^2}\left( \frac{1}{n_1}-\frac{1}{N}\right)s_A^2 & =
\frac{N}{n_1n_0}s_A^2
 = \frac{N}{n_1n_0}(s_Y^2  + \beta^2 s_D^2 - 2\beta s_{YD}).
\end{align*}
Moreover, as commented in the Supplementary Material, Condition \eqref{eq:condi_CLThajek} for $\{A_i:i=1,\ldots,N\}$ holds for any $\beta$, if $\{Y_i:i=1,\ldots,N\}$ and $\{D_i:i=1,\ldots,N\}$ satisfy 
\begin{align}\label{eq:condi_instru_Y_D}
\frac{1}{\min(n,N-n)} 
\left\{
\frac{\max_{1\leq i\leq N} (Y_i-\bar{Y})^2}{s_Y^2-s_{YD}^2/s_D^2}
+
\frac{\max_{1\leq i\leq N} (D_i-\bar{D})^2}{s_D^2-s_{YD}^2/s_Y^2}
\right\}
 \rightarrow 0.
\end{align}
Let $\Phi(\cdot)$ be the cumulative distribution function of the standard Normal random variable.
Based on the Normal approximation, the $1-\alpha$ confidence interval for $\beta$ is the values satisfying $ |   \widehat{\tau}_A /\sqrt{ \Var_0(\widehat{\tau}_A)   }  |\leq |    \Phi^{-1}(\alpha/2)   |$, or equivalently the solution of the following inequality:
\begin{eqnarray}\label{eq::ci-fieller-creasy}
(\widehat{\tau}_Y - \beta \widehat{\tau}_D)^2 \leq \{  \Phi^{-1}(\alpha/2) \}^2\times   \frac{N}{n_1 n_0} (s_Y^2  + \beta^2 s_D^2 - 2\beta s_{YD}).
\end{eqnarray}
Note that for different observed data, the solution of the quadratic inequality in \eqref{eq::ci-fieller-creasy} can be an interval, an empty set, or two disjoint sets, a surprising phenomenon that also occured in the classical Fieller--Creasy problem \citep{fieller1954some, creasy1954limits}. Let 
$\eta = {N}/{(n_1 n_0)} \cdot\{  \Phi^{-1}(\alpha/2) \}^2$ and 
$\Delta = 4(\widehat{\tau}_D\widehat{\tau}_Y -  \eta s_{YD})^2 - 4(\widehat{\tau}_D^2 - \eta s_D^2)(\widehat{\tau}_Y^2 - \eta s_Y^2)$. 
Table \ref{tab:form_ci_beta_fewer} shows four possible forms of the confidence sets for $\beta$. We give more detailed discussion in the Supplementary Material. 
\begin{table}[htbp]
\centering
\caption{Four possible forms of the confidence sets for $\beta$, where $c_1<c_2$ denote two roots of \eqref{eq::ci-fieller-creasy} if they exist.
}\label{tab:form_ci_beta_fewer}
\begin{tabular}{ccc}
\toprule
$\widehat{\tau}_D^2 - \eta s_D^2$ & $\Delta$ & form of the confidence set for $\beta$\\
\midrule
$> 0$ & $<0$ & empty set \\
$> 0$ & $>0$ & $[c_1, c_2]$ \\
$<0$ & $\leq 0$ & the whole real line\\
$<0$ & $>0$ & $(-\infty, c_1]\bigcup [c_2, \infty)$
\\
\bottomrule
\end{tabular}
\end{table}
$\hfill \square$
\end{example}

\subsection{Random partition}
\label{sec::partition}

Theorem \ref{thm:CLThajek} is useful for deriving asymptotic distributions in survey sampling and treatment-control experiments. However, it is not adequate for treatments with more than two levels. In this section, we present a CLT for random partition as an extension of \citet[][Theorem 19]{lehmann2006nonparametrics}. 

Again let $\Pi_N = \{y_{N1}, \ldots, y_{NN}\}$ be a finite population. Similar to Section \ref{sec:SRS}, we define $\bar{y}_N$, $v_N$, and $m_N$ as the finite population mean, variance, and the maximum squared distance of the $y_{Ni}$'s from the population mean.
We consider a random partition of $\Pi_N$: units are partitioned into $Q$ groups of size $(n_1,\ldots, n_Q)$, where $\sum_{q=1}^Q n_q = N$. Let $L_i$ be the group number, where $L_i = q$ if unit $i$ belongs to group $q.$ The group number vector is $ (L_1, \ldots, L_N)$, and the probability that $(L_1, \ldots, L_N)$ takes a particular value $ (l_1,\ldots,l_N)$ is $n_1!\cdots n_Q!/N!$, where $\sum_{i=1}^N 1\{ l_i=q\}=n_q$ for all $q$.
For any $1\leq q\le Q$, the sample average in group $q$ is 
$
\ybarsq=\sum_{i:L_i=q}y_{Ni}/n_q.
$ 
Because group $q$ is a simple random sample with size $n_q$, the sample average has mean $\bar{y}_N$ and variance $(n_q^{-1}-N^{-1})v_N$.
Instead of focusing on only one sample average as in Theorem \ref{thm:CLThajek}, we consider the joint distribution of $Q$ standardized sample averages
\begin{align}\label{eq:t_N_sim_clt}
\bm{t}_N = \left(
\frac{\bar{y}_{\text{S}1}-E(\bar{y}_{\text{S}1})}{\sqrt{\Var(\bar{y}_{\text{S}1})}},
\ldots, \frac{ \bar{y}_{\text{S}Q} -E( \bar{y}_{\text{S}Q} )}{\sqrt{\Var( \bar{y}_{\text{S}Q} )}}
\right)^\top.
\end{align}
\begin{proposition}\label{prop:var_tN}
$\bm{t}_N$ has mean zero and covariance matrix
\begin{align}\label{eq:var_t_N}
\Cov(\bm{t}_N) & =
\begin{pmatrix}
1 & -\sqrt{\frac{n_1n_2}{(N-n_1)(N-n_2)}} & \cdots & -\sqrt{\frac{n_1n_Q}{(N-n_1)(N-n_Q)}}\\
-\sqrt{\frac{n_2n_1}{(N-n_2)(N-n_1)}} & 1 & \cdots & -\sqrt{\frac{n_2n_Q}{(N-n_2)(N-n_Q)}}\\
\vdots & \vdots & \ddots & \vdots\\
-\sqrt{\frac{n_Qn_1}{(N-n_Q)(N-n_1)}} & -\sqrt{\frac{n_Qn_2}{(N-n_Q)(N-n_2)}} & \cdots & 1
\end{pmatrix}.
\end{align}
\end{proposition}
Proposition \ref{prop:var_tN} appeared in \citet[][page 393]{lehmann2006nonparametrics}. Furthermore, $\bm{t}_N$ is asymptotically Normal under the regularity condition below.

\begin{theorem}\label{thm:sim_clt}
Let $(\bar{y}_{\text{S}1}, \ldots, \bar{y}_{\text{S}Q})$ be the $Q$ sample averages of a random partition of sizes $(n_1, \ldots, n_Q)$ for a finite population $\Pi_N =  \{y_{N1}, \ldots, y_{NN}\}$. 
As $N\rightarrow \infty$, if
(i) $\Cov(\bm{t}_N)$ in \eqref{eq:var_t_N} has a limiting value $\bm{V}\in \mathbb{R}^{Q\times Q}$, and (ii)
\begin{align}\label{eq:regu_sim_clt}
\frac{1}{\min_{1\leq q\leq Q} n_q} \cdot
\frac{m_N}{v_N} & \rightarrow 0,
\end{align}
then
$
\bm{t}_N \converged \mathcal{N}\left(
\bm{0}, \bm{V}
\right).
$
\end{theorem}

Because the components of $\bm{t}_N$ are linearly dependent, the rank of its covariance matrix is $Q-1.$ When $Q=2$, Theorem \ref{thm:sim_clt} reduces to Theorem \ref{thm:CLThajek}. We use \citet{fraser1956vector}'s vector CLT for rank statistics to prove Theorem \ref{thm:sim_clt}.  \citet[][Theorem 19, page 393]{lehmann2006nonparametrics} presented a slightly weaker form and gave a different proof, requiring an equivalent form of Condition \eqref{eq:regu_sim_clt} and additionally that $n_q\rightarrow\infty$ and $n_q/N$ has a limiting value less than 1. Recall that $m_N/v_N\geq 1-N^{-1}$, and therefore Condition \eqref{eq:regu_sim_clt} implies $n_q\rightarrow \infty$ for all $q$.

Theorem \ref{thm:sim_clt} is particularly useful for studying the asymptotic properties of randomization tests in completely randomized experiments with multiple arms. 
Consider a completely randomized experiment with $N$ units and $Q$ treatments. For each unit $i$, the $Q$ dimensional vector 
$\left(
Y_i(1), \ldots, Y_i(Q)
\right)$ denotes its potential outcomes under all treatments. Let $L_i$ be the treatment number for unit $i$, where $L_i = q$ if it is assigned to treatment $q$. Therefore, $Y_i=Y_i(L_i)$ is the observed outcome of unit $i$. Fisher's sharp null hypothesis states that
\begin{align}
H_0: Y_i(1) = Y_i(2) = \cdots =  Y_i(Q) \quad (i=1,\ldots, N).
\label{eq:sharp_null_rank}
\end{align}
Under the sharp null hypothesis that the treatment does not affect any units, all the observed outcomes are fixed numbers, and the randomization of the treatment numbers are the only source of randomness. Because the joint distribution of the $L_i$'s is known, the distribution of any test statistic under $H_0$ is also known and can often be approximated by simple distributions with large sample sizes. We review three examples for testing the sharp null hypothesis using the ranks of the pooled observed outcomes.
Assuming no ties, let $R_i$ be the rank of ${Y}_i$ among all units, $\bar{R}_{(q)}=\sum_{i: L_i = q}R_{i}/n_q$ be the average rank of units under treatment $q$, and 
$$
\widetilde{R}_{(q)} = 
\frac{\bar{R}_{(q)} - E(\bar{R}_{(q)})}{\sqrt{\Var(\bar{R}_{(q)})}}
=
\sqrt{\frac{12n_q}{(N+1)(N-n_q)}}\left(\bar{R}_{(q)}-\frac{N+1}{2}\right)
$$ 
be the standardized rank average.
\begin{corollary}\label{cor::rank}
Under the sharp null hypothesis in \eqref{eq:sharp_null_rank}, as $N\rightarrow\infty$, if for each $1\leq q\leq Q$,
$n_q\rightarrow\infty$ and $n_q/N\rightarrow \gamma_q<1$, then
\begin{align}\label{eq::rank}
 \left(
\widetilde{R}_{(1)}, \widetilde{R}_{(2)}, \ldots,
\widetilde{R}_{(Q)}
\right)^\top \converged \mathcal{N}(\bm{0}, \bm{V}_R),
\end{align}
where $\bm{V}_R$ is a correlation matrix with the $(q,r)$th entry $-\sqrt{{\gamma_q\gamma_r}/\{(1-\gamma_q)(1-\gamma_r)\}}$.
\end{corollary}

Corollary \ref{cor::rank} plays a crucial role in nonparameteric tests based on ranks. Below we discuss three examples. 

\begin{example}[Krushal--Wallis test]\label{eg:krushal_walls}
Conducting analysis of variance on the ranks results in the Kruskal--Wallis test statistic
\begin{align*}
H & =
(N-1) 
\frac{\sum_{q=1}^Q n_q\{\bar{R}_{(q)} - \bar{R}\}^2}{\sum_{i=1}^N ({R}_i - \bar{R})^2}
= \sum_{q=1}^Q \frac{N-n_q}{N}
\widetilde{R}_{(q)}^2,
\end{align*}
which is a quadratic form of the standardized ranks in \eqref{eq::rank}.
Corollary \ref{cor::rank}
and the properties of quadratic forms of multivariate Normal distributions guarantee that $H$ converges to a $\chi^2_{Q-1}$ random variable, as shown in \citet[][Appendix 8, Example 31]{lehmann2006nonparametrics}.
$\hfill \square$
\end{example}

\begin{example}
Besides the Kruskal--Wallis test statistic,  \citet{lehmann2006nonparametrics} suggested using
$
\max_{1\leq q \leq Q}\bar{R}_{(q)}
$ 
for testing \eqref{eq:sharp_null_rank}. Another reasonable test statistic is 
$
\max_q \bar{R}_{(q)} - \min_q \bar{R}_{(q)} = \max_{q,r} \{\bar{R}_{(q)}   - \bar{R}_{(r)} \}.
$
After proper standardization, 
the asymptotic distributions of both test statistics can be approximated by the distribution functions of 
multivariate Normal distributions.
$\hfill \square$
\end{example}

\begin{example}[Rank test with dose]
Assume that each level of the treatment represents a dose $z_q$ for $q=1,\ldots, Q.$ If we anticipate a monotonic dose-response relationship, 
then it seems more reasonable to use the following test statistic $\sum_{q=1}^Q z_q \bar{R}_{(q)}$, weighting the average ranks by the corresponding dose \citep[c.f.][]{page1963, Rosenbaum01012003}. 
Because it is a linear combination of \eqref{eq::rank}, 
Corollary \ref{cor::rank} implies that  
this test
statistic has an
asymptotic Normal distribution,
based on which we can conduct one sided or two sided tests.
$\hfill \square$
\end{example}

\section{Finite population CLTs in randomized experiments}\label{sec:fclt_cre}

In this section, we will establish finite population CLTs in completely randomized experiments without assuming the sharp null hypothesis.
These CLTs play crucial roles in repeated sampling evaluations of causal effect estimators in randomization-based causal inference. These finite population CLTs deal with vector outcomes under multiple treatments, and work for general causal estimators. 
Consider an experiment with $N$ units and $Q$ treatments, where $n_q$ units receive treatment $q$, and $\sum_{q=1}^Q n_q = N$.
For any treatment $1\leq q\leq Q$, let $\bm{Y}_i(q)\in \mathbb{R}^p$ be the $p$ dimensional potential outcome vector of unit $i$, and 
$\bar{\bm{Y}}(q) = \sum_{i=1}^N \bm{Y}_i(q)/N \in \mathbb{R}^p$ be the average potential outcome vector of the $N$ units. 
Causal estimands of interest are often linear combinations of the potential outcomes. 
For any $K\geq 1$ and  coefficient matrices $\bm{A}_q \in \mathbb{R}^{K \times p}$ $(q=1,2,\ldots,Q)$, 
we consider individual causal effect of the form $\bm{\tau}_i(\bm{A})=\sum_{q=1}^Q \bm{A}_q \bm{Y}_i(q)$, and the population average causal effect of the form
\begin{eqnarray}\label{eq:def_tau_A}
\bm{\tau}(\bm{A}) 
= \left(
{\tau}_{(1)}(\bm{A}), \ldots, {\tau}_{(K)}(\bm{A}) 
\right)^\top 
= \frac{1}{N} \sum_{i=1}^N \bm{\tau}_i(\bm{A}) = \sum_{q=1}^Q \bm{A}_q \bar{\bm{Y}}(q).
\label{eq::general-causal}
\end{eqnarray}
The general causal estimand \eqref{eq::general-causal} covers many important cases, including vector outcomes and joint effects. For example, if 
$
\bm{A}_1 = \bm{I}_{p\times p}, \bm{A}_2 = -\bm{I}_{p\times p}, \bm{A}_3 = \cdots = \bm{A}_Q = \bm{0}_{p\times p},
$
then 
${\bm{\tau}}(\bm{A}) = \bar{\bm{Y}}(1)-\bar{\bm{Y}}(2)$ is the average causal effect comparing treatments 1 and 2. 
If 
\begin{align*}
\bm{A}_1 =
\begin{pmatrix}
\bm{I}_{p\times p}\\
\bm{I}_{p\times p}
\end{pmatrix},
\quad 
\bm{A}_2 =
\begin{pmatrix}
-\bm{I}_{p\times p}\\
\bm{0}_{p\times p}
\end{pmatrix},
\quad 
\bm{A}_3 =
\begin{pmatrix}
\bm{0}_{p\times p}\\
-\bm{I}_{p\times p}
\end{pmatrix},
\quad 
\bm{A}_4 =
\cdots =
\bm{A}_Q =
\begin{pmatrix}
\bm{0}_{p\times p}\\
\bm{0}_{p\times p}
\end{pmatrix},
\end{align*}
then 
$
\bm{\tau}(\bm{A}) = 
\left(
\bar{\bm{Y}}(1)-\bar{\bm{Y}}(2), 
\bar{\bm{Y}}(1)-\bar{\bm{Y}}(3)
\right)
$
is the average causal effects comparing treatment 1 to treatments 2 and 3. Many applications are interested in jointly estimating multiple causal effects. 
We will discuss more examples intensively in Sections \ref{sec::applications} and \ref{sec::factorial}.

We consider again a completely randomized experiment with $Q$ treatment groups of sizes $(n_1,\ldots,n_Q)$, as introduced in Section \ref{sec::partition}.
Let $\bm{Y}_i = \bm{Y}_{i}(L_i)$ be the observed outcome of unit $i$. The average observed outcome under treatment $q$ is
$
\widehat{\bar{\bm{Y}}}(q)= \sum_{i:L_i = q} \bm{Y}_i/n_q,
$ 
and an intuitive estimator for $\bm{\tau}(\bm{A})$ is
\begin{eqnarray*}
\widehat{\bm{\tau}}(\bm{A}) = 
(\widehat{{\tau}}_{(1)}(\bm{A}), \ldots, \widehat{{\tau}}_{(K)}(\bm{A}))^\top
=\sum_{q=1}^Q \bm{A}_q \widehat{\bar{\bm{Y}}}(q),
\end{eqnarray*} 
by replacing $\bar{\bm{Y}}(q)$ by $\widehat{\bar{\bm{Y}}}(q)$ in \eqref{eq:def_tau_A}.
A central question is to study the repeated sampling properties of $\widehat{\bm{\tau}}(\bm{A}) $ over all randomizations.

Extending \citet{Neyman:1923}, the following theorem shows that $\widehat{\bm{\tau}}(\bm{A})$ is unbiased for $\bm{\tau}(\bm{A})$, with sampling variance depending on the finite population covariances of the potential outcomes 
$$
\bm{S}^2_{q} = \frac{1}{N-1}
\sum_{i=1}^N 
\left\{
\bm{Y}_{i}(q)-\bar{\bm{Y}}(q)
\right\}
\left\{
\bm{Y}_{i}(q)-\bar{\bm{Y}}(q)
\right\}^\top, \quad (q=1,\ldots,Q)
$$
the finite population covariances between the potential outcomes 
\begin{align*}
\bm{S}_{qr}= \frac{1}{N-1}
\sum_{i=1}^N 
\left\{
\bm{Y}_{i}(q)-\bar{\bm{Y}}(q)
\right\}
\left\{
\bm{Y}_{i}(r)-\bar{\bm{Y}}(r)
\right\}^\top, \quad (q,r=1,\ldots,Q; q\neq r)
\end{align*}
and the finite population covariance of the individual causal effects 
\begin{align*}
\bm{S}^2_{\bm{\tau}(\bm{A})} & = \frac{1}{N-1}
\sum_{i=1}^N \left\{
\bm{\tau}_i(\bm{A})- \bm{\tau}(\bm{A})
\right\}
\left\{
\bm{\tau}_i(\bm{A})- \bm{\tau}(\bm{A})
\right\}^\top.
\end{align*}

\begin{theorem}\label{thm::repeated-sampling}
In a completely randomized experiment with $n$ units and $Q$ groups, let $\bm{Y}_i(q) \in \mathbb{R}^p$ be unit $i$'s potential outcome under treatment $q.$
Over all $N!/(n_1!\cdots n_Q!)$
randomizations, the estimator $\widehat{\bm{\tau}}(\bm{A})$ has mean ${\bm{\tau}}(\bm{A})$ and covariance
\begin{align*}
\Cov\left\{
\widehat{\bm{\tau}}(\bm{A})
\right\}  = \sum_{q=1}^Q
\frac{1}{n_q}
\bm{A}_q
\bm{S}^2_q
\bm{A}_q^\top - \frac{1}{N} \bm{S}^2_{\bm{\tau}(\bm{A})}.
\end{align*}
\end{theorem}

To construct a confidence set for the causal estimand ${\bm{\tau}}(\bm{A})$, we need to establish CLTs under complete randomization. 
Below we use $[\bm{Y}]_{(k)}$ to denote the $k$-th coordinate of a vector $\bm{Y}.$
Analogous to Theorems \ref{thm:CLThajek} and \ref{thm:sim_clt}, the asymptotic behavior of $\widehat{\bm{\tau}}(\bm{A})$ depends on the maximum square distance of the $k$-th coordinate of the $\bm{A}_q\bm{Y}_i(q)$'s from their population mean
$$
m_{q}(k) = \max_{1\leq i \leq N}\left[
\bm{A}_q\bm{Y}_i(q) - \bm{A}_q\bar{\bm{Y}}(q)
\right]_{(k)}^2,
 \quad (1\leq k\leq K) 
$$
the finite population variance of the $k$-th coordinate of the $\bm{A}_q\bm{Y}_i(q)$'s
$$
v_{q}(k) = \frac{1}{N-1}\sum_{i=1}^N \left[
\bm{A}_q\bm{Y}_i(q) - \bm{A}_q\bar{\bm{Y}}(q)
\right]_{(k)}^2,
\quad (1\leq k\leq K)
$$
and the finite population variance of the $k$-th coordinate of the $\bm{\tau}_i(\bm{A})$'s
$$
v_{\bm{\tau}}(k) = \frac{1}{N-1}\sum_{i=1}^N \left[
\bm{\tau}_i(\bm{A}) - \bm{\tau}(\bm{A})
\right]_{(k)}^2,
\quad (1\leq k\leq K)  . 
$$

\begin{theorem}\label{thm:fclt_cre}
Under the setting of Theorem \ref{thm::repeated-sampling}, as $N\rightarrow\infty,$
if
\begin{align}\label{eq:condi_cre}
\max_{1\leq q\leq Q}\max_{1\leq k\leq K}\frac{1}{n_q^2}\frac{m_q(k)}{\sum_{r=1}^Qn_r^{-1}v_r(k)-N^{-1}v_{\bm{\tau}}(k)}\rightarrow 0,
\end{align}
and the correlation matrix of $\widehat{\bm{\tau}}(\bm{A})$  has a limiting value $\bm{V}$, then
\begin{align}\label{eq:fclt_cre}
\left(
\frac{\widehat{{\tau}}_{(1)}(\bm{A})-{{\tau}}_{(1)}(\bm{A})}{\sqrt{\Var\left\{\widehat{{\tau}}_{(1)}(\bm{A})\right\}}}, \ldots, 
\frac{\widehat{{\tau}}_{(K)}(\bm{A})-{{\tau}}_{(K)}(\bm{A})}{\sqrt{\Var\left\{\widehat{{\tau}}_{(K)}(\bm{A})\right\}}}
\right) \converged \mathcal{N}\left(
\bm{0}, \bm{V}
\right).
\end{align}
\end{theorem}

Although Condition \eqref{eq:condi_cre} is general, it is not intuitive and needs more explanation. We consider two special cases and provide more easy-to-check conditions for the CLT of $\widehat{\bm{\tau}}(\bm{A})$.
First,
we assume that the causal effects are additive, i.e., $\bm{\tau}_i(\bm{A})$ is a constant vector for all unit $i$. The following corollary is useful for randomization inference with the additive causal effects assumption, including randomization tests under the sharp null hypothesis.

\begin{corollary}\label{cor:add_causal_effects}
Under the setting of Theorem \ref{thm::repeated-sampling} with the additive causal effect assumption, as $N\rightarrow\infty,$
if
\begin{align}\label{eq:condition_exp_add}
\max_{1\leq q\leq Q}\max_{1\leq k\leq K}\frac{1}{n_q}\frac{m_q(k)}{v_q(k)}\rightarrow 0,
\end{align}
and the correlation matrix of $\widehat{\bm{\tau}}(\bm{A})$  has a limiting value $\bm{V}$, then \eqref{eq:fclt_cre} holds.
\end{corollary}

Corollary \ref{cor:add_causal_effects} is similar to Theorem \ref{thm:CLThajek} in the sense that the regularity condition involves the ratio between the maximum squared distance and the finite population variance of certain populations. 
It
is directly implied by Theorem \ref{thm:fclt_cre} by noticing that $v_{\bm{\tau}}(k)=0$ under the additive causal effect assumption. If we use the ranks of the observed outcomes in nonparametric tests for the sharp null hypothesis, then $[\bm{A}_q\bm{Y}_i(q)]_{(k)}$'s are a permutation of $\{1,2,\ldots,N\}$, and the corresponding regularity condition holds automatically, because as $n_q\rightarrow \infty$ for all $1\leq q\leq Q$, 
\begin{eqnarray}
\frac{ m_q(k) } { n_q v_q(k) }=  \frac{ 3(N-1)^2 } {  n_qN(N+1) }\rightarrow 0.
\label{eq::rank-condition}
\end{eqnarray}

Second, we assume that the finite population of experimental units has limiting covariances, and the proportions of units receiving all treatments have positive limiting values.

\begin{theorem}\label{cor:fclt_exp_stable}
Under the setting of Theorem \ref{thm::repeated-sampling}, if for any $1\leq q\neq r\leq Q$, 
$\bm{S}^2_{q}$ and $\bm{S}_{qr}$ have limiting values,  $n_q/N$ has positive limiting value, and 
$
\max_{1\leq q\leq Q}\max_{1\leq i \leq N}\left\|\bm{Y}_i(q) - \bar{\bm{Y}}(q)\right\|_2^2/N \rightarrow 0,
$
then $N\Var\left\{
\widehat{\bm{\tau}}(\bm{A}) 
\right\}$ has a limiting value, denoted by $\bm{V}$, and 
$$
\sqrt{N}\left\{ \widehat{\bm{\tau}}(\bm{A}) - \bm{\tau}(\bm{A}) \right\} \converged \mathcal{N}(\bm{0},\bm{V}).
$$
\end{theorem}

Note that by properly scaling the potential outcomes, the diagonal elements of the finite population covariances $\bm{S}_q^2$'s can always have limits. Thus the condition in Theorem  \ref{cor:fclt_exp_stable} that the $\bm{S}^2_{q}$'s and $\bm{S}_{qr}$'s have limits essentially require only the convergence of the correlations between different coordinates of potential outcomes. 
The regularity conditions in Theorem \ref{cor:fclt_exp_stable} also involve the restriction on the order of the maximum squared distance of the $\bm{Y}_i(q)$'s from the population mean $\bar{\bm{Y}}(q)$.
When the coordinates of the $\bm{Y}_i(q)$'s are bounded, or i.i.d draws from a superpopulation with more than two moments, 
 the regularity condition $\max_{1\leq i \leq N}\left\|\bm{Y}_i(q) - \bar{\bm{Y}}(q)\right\|_2^2/N \rightarrow 0$ holds with probability one. 


For a scalar outcome, \citet{freedman2008regression_a,freedman2008regression_b} established a finite population CLT under stronger conditions, requiring that the fourth moments of the potential outcomes are finite. Theorem \ref{thm:fclt_cre} deals with a vector outcome and requires weaker conditions.

Now we consider estimation of the covariance of $\widehat{\bm{\tau}}(\bm{A}).$
Because treatment arm $q$ is a simple random sample of the finite population, the sample covariance of $\bm{Y}_i$ under treatment $q$,
\begin{align*}
\bm{s}^2_{q} = \frac{1}{n_q-1}\sum_{L_i=q}\left\{
\bm{Y}_i - \widehat{\bar{\bm{Y}}}(q)
\right\}\left\{
\bm{Y}_i - \widehat{\bar{\bm{Y}}}(q)
\right\}^\top, \quad (1\leq q\leq Q)
\end{align*}
is unbiased for the population covariance $\bm{S}^2_{q}$ \citep{cochran1977}.
However, $\bm{S}^2_{\bm{\tau}(\bm{A})}$ is generally not estimable, because the potential outcomes $\bm{Y}_i(1), \ldots, \bm{Y}_i(Q)$ cannot be jointly observed. On average, the covariance estimator
$
\widehat{\bm{V}}_{\bm{A}} = \sum_{q=1}^Q
n_q^{-1}
\bm{A}_q
\bm{s}^2_{q}
\bm{A}_q^\top
$
over estimates the sampling variance by $\bm{S}^2_{\bm{\tau}(\bm{A})}/N$. Let 
$q_{K, 1-\alpha}$ 
be the $(1-\alpha)$th quantile of a $\chi^2$ distribution with degrees of freedom $K$.
\begin{proposition}\label{prop:con_set_tau}
Under the regularity conditions in Theorem \ref{cor:fclt_exp_stable},  $\bm{s}^2_q - \bm{S}^2_q \overset{p}{\longrightarrow} 0$ for each $1\leq q\leq Q$. If the limits of $\bm{S}^2_q$'s are not all zero, 
then the probability that $\widehat{\bm{V}}_{\bm A}$ is nonsingular converges to one, and 
the Wald-type confidence region for $\bm{\tau}(\bm{A})$,
$$
\left\{\bm{\mu}: \  \left\{ \widehat{\bm{\tau}}(\bm{A}) - \bm{\mu} \right\}^\top 
\widehat{\bm{V}}_{\bm{A}}^{-1}
\left\{ \widehat{\bm{\tau}}(\bm{A}) - \bm{\mu} \right\} \leq q_{K, 1-\alpha}\right\},
$$
has asymptotic coverage rate at least as large as $1-\alpha,$ and the asymptotic coverage rate equals $1-\alpha$ if the causal effects are additive.
\end{proposition}

Theorems \ref{thm::repeated-sampling}--\ref{cor:fclt_exp_stable} and Proposition \ref{prop:con_set_tau} generalize \citet{Neyman:1923}.
For a binary treatment and a scalar outcome, his results \citep[c.f.][]{imbens2015causal} ensure that the difference-in-means estimator $\widehat{\tau}$ is unbiased for $\tau = \sum_{i=1}^N \{  Y_i(1) - Y_i(0) \}  /N$ with variance $S_1^2/n_1+S_0^2/n_0 - S_\tau^2/N,$ where $S_1^2$ and $S_0^2$ are the finite population variances of the treatment and control potential outcomes, and $S_\tau^2$ is the finite population  variance of the individual causal effects. Based on the Normal approximation, a conservative $1-\alpha$ confidence interval for $\tau$ is $\widehat{\tau} \pm \Phi^{-1}(1-\alpha/2) (s_1^2/n_1 + s_0^2/n_0)^{1/2}$, where $s_1^2$ and $s_0^2$ are the sample variances of the outcomes under treatment and control.
\citet{aronow2014} proposed a consistent estimator of sharp bounds on the sampling variance of difference-in-means estimator in this setting with a scalar outcome. It will be interesting to extend their result to general experiments with general outcomes.

\section{Applications to treatment-control experiments}\label{sec::applications}

The generality of Theorems \ref{thm:fclt_cre} and \ref{cor:fclt_exp_stable} allows us to prove asymptotics for many causal inference problems. Below we review five important examples in treatment-control experiments. Previous literature provided intuitive arguments for asymptotic Normalities in these examples, but our proofs based on Theorems \ref{thm:fclt_cre} and \ref{cor:fclt_exp_stable} are more rigorous and provide more general results.

In Examples  \ref{eg:more_powerful_test}--\ref{eg:regression_adjust}, 
we consider completely randomized treatment-control experiments. For descriptive simplicity, we unify the notation in these four examples. 
Consider a completely randomized experiment with $N$ units, among which $n_1$ units receive treatment and $n_0$ receive control. For each unit $i$, let $Z_i$ be the treatment assignment indicator ($Z_i=1$ if treatment; $Z_i=0$ if control), 
$\bm{X}_i=(X_{1i}, \ldots, X_{Ki})$  the $K$ dimensional pretreatment covariates, 
$Y_i(z)$ the potential outcome under treatment arm $z$, 
$\tau_i=Y_i(1)-Y_i(0)$ the individual causal effect, 
and $Y_i=Z_iY_i(1)+(1-Z_i)Y_i(0)$ the observed outcome. We use $\bar{Y}(z)$ 
and $\bar{\bm{X}}=(\bar{X}_{1}, \ldots, \bar{X}_{K})$ 
to denote finite population means of the potential outcomes
and covariates, 
and $\tau=\bar{Y}(1)-\bar{Y}(0)$ to denote the average causal effect. 
To facilitate the discussion, we center the covariates with zero finite population means ($\bar{\bm{X}}=\bm{0}$).
As $N\rightarrow\infty$, 
we assume the proportions of units receiving both treatments have positive limiting values, the finite population variances and covariances among potential outcomes and covariates have limiting values, and
\begin{align}\label{eq:max_condi_eg}
\frac{1}{N}\max_{1\leq i \leq N}\left\{
Y_i(z) - \bar{Y}(z)
\right\}^2 \rightarrow 0, \ \ 
\frac{1}{N}\max_{1\leq i \leq N}
X_{ki}^2 \rightarrow 0, \quad (z=0,1; k=1,\ldots,K).
\end{align}





\begin{example}[Combining test statistics in a randomization test]\label{eg:more_powerful_test}
For testing the sharp null hypothesis that
$
Y_i(1)=Y_i(0)
$
for all $i$, 
we need to choose a test statistic, and can use the finite population central limit theorem to determine the rejection region. Two commonly used statistics are the difference-in-means statistic
$$
T  = \frac{1}{n_1}\sum_{i=1}^N Z_iY_i-\frac{1}{n_0}\sum_{i=1}^N (1-Z_i)Y_i,
$$
and the Wilcoxon rank sum statistic
$$
W  = \frac{1}{n_1}\sum_{i=1}^N Z_i R_i-\frac{1}{n_0}\sum_{i=1}^N (1-Z_i)R_i,
$$
where $R_i$ is the rank of $Y_i$ among all units. 
Combining $T$ and $W$ can sometimes leads to a more powerful test than using only one of them. To determine the rejection region, it is important to derive the asymptotic joint distribution of $(T,W)$ under complete randomization. 
Under the sharp null hypothesis, all the $Y_i$'s and $R_i$'s are fixed quantities unaffected by the treatment, and the sampling distributions of $T$ and $W$
are determined by the distribution of $(Z_1,\ldots, Z_N)$.
Let $\bar{Y}$ and $\bar{R}$ be the averages, and 
$s_Y^2$ and $s_R^2$ be the variances of pooled $Y_i$'s and $R_i$'s. Let
$\bm{V}_{\rho}$ be a two dimensional correlation matrix with off-diagonal element $\rho$.
Using Corollary \ref{cor:add_causal_effects} with potential outcomes $(Y_i,R_i)^\top$ under both treatment conditions and coefficient matrices $\bm{A}_1=\bm{I}_{2\times 2}$ and  $\bm{A}_0=-\bm{I}_{2\times 2}$, if 
\begin{align}\label{eq:condi_more_powerful}
 \frac{\max_{1\leq i \leq N}\left|
Y_i-\bar{Y}
\right|^2}{n_z s_{Y}^2} \rightarrow 0, \ \ 
\frac{\max_{1\leq i \leq N}\left|
R_i-\bar{R}
\right|^2}{n_z s_{R}^2} \rightarrow 0, \quad (z=0,1)
\end{align}
and the finite population correlation between $Y_i$ and $R_i$, $\rho_N$,  has a limiting value $\rho_{\infty}$,  then
\begin{align*}
\left(
\frac{T}{\sqrt{\Var(T)}}, \frac{W}{\sqrt{\Var(W)}}
\right) 
= \sqrt{ \frac{n_1n_0}{N} }
\left(
T/s_Y, W/s_R
\right)
\converged \mathcal{N}\left(
\bm{0}, \bm{V}_{\rho_{\infty}}
\right).
\end{align*}
Note that the first condition in \eqref{eq:condi_more_powerful} is ensured by \eqref{eq:max_condi_eg}, and
the second condition in \eqref{eq:condi_more_powerful} holds automatically as argued in \eqref{eq::rank-condition}.
To determine the rejection region for $(T,W)$, similar to \citet{rosenbaum2012testing}, we introduce $1-\gamma_{\rho}(c)$, the probability of the $2$-dimensional lower orthant $(-\infty, c]\times (-\infty, c]$ for a $2$-dimensional Normal distribution with mean zero and covariance matrix $\bm{V}_{\rho}$.
Let $c_\alpha$ be a constant such that $\gamma_{\rho_N}(c_\alpha)=\alpha$. The rejection region with significance level $\alpha$ is
$
\sqrt{  n_1n_0 / N }
\max\left\{ 
T/s_Y, W/s_R
\right\}>c_\alpha.
$
We can easily generalize the above results to more than two test statistics. 
$\hfill \square$
\end{example}

\begin{example}[Multiple randomization tests]\label{eg::multiple_test} 
Assuming for each unit $i$ under treatment assignment $z$, there is a two dimensional potential outcome vector $\bm{Y}_i(z)=(Y_{1i}(z), Y_{2i}(z))$.
We are interested in testing two sharp null hypotheses that $Y_{1i}(1)=Y_{1i}(0)$ and $Y_{2i}(1)=Y_{2i}(0)$ for all $i$, simultaneously.
Let $Y_{1i}$ and $Y_{2i}$  be the observed outcomes for unit $i$.
We can use the difference-in-means statistics to test both sharp null hypotheses:
\begin{align*}
T_1 & = \frac{1}{n_1}\sum_{i=1}^N Z_iY_{1i} - \frac{1}{n_0}\sum_{i=1}^N (1-Z_i)Y_{1i},\\
T_2 & = \frac{1}{n_1}\sum_{i=1}^N Z_iY_{2i} - \frac{1}{n_0}\sum_{i=1}^N (1-Z_i)Y_{2i}.
\end{align*}
In Example \ref{eg:more_powerful_test}, the two test statistics are used to test the same sharp null hypothesis, but
$T_1$ and $T_2$ here are used to test two different null hypotheses. 
One way to control the type one error is the Bonferroni correction, using only the marginal asymptotic distribution of $T_1$ and $T_2$. A more efficient way is to use the joint distribution of $(T_1,T_2)$.
For $k=1,2$, 
let $\bar{Y}_k$ and $s_{Y_k}^2$ be the mean 
and variance of the pooled observed values of the $Y_{ki}$'s.
We use Corollary \ref{cor:add_causal_effects} with potential outcomes $(Y_{1i}(z), Y_{2i}(z))^\top$ and coefficient matrices $\bm{A}_1=\bm{I}_{2\times 2}$ and $\bm{A}_0=-\bm{I}_{2\times 2}$.
According to Corollary \ref{cor:add_causal_effects}, under the sharp null hypothesis, if 
$
\max_{1\leq i \leq N} ( 
Y_{ki}-\bar{Y}_k ) ^2 / (  n_zs_{Y_k}^2) \rightarrow 0 
$
for $k=1,2$ and $z=0,1$,
and the finite population correlation between $Y_{1i}$ and $Y_{2i}$, $\eta_N$, has a limiting value $\eta_{\infty}$, 
then 
\begin{align*}
\left(
\frac{T_1}{\sqrt{\Var(T_1)}}, \frac{T_2}{\sqrt{\Var(T_2)}}
\right)
=\sqrt{\frac{n_1n_0}{N}}\left(
T_1/s_{Y_1}, T_2/s_{Y_2}
\right)
\converged
\mathcal{N}
\left(
\bm{0}, \bm{V}_{\eta_{\infty}}
\right),
\end{align*}
recalling that $\bm{V}_{\eta_{\infty}}$ is a two dimensional correlation matrix with off-diagonal element $\eta_{\infty}$. Let $c_\alpha$ be a constant satisfying $\gamma_{\eta_N}(c_\alpha)=\alpha$ defined in Example \ref{eg:more_powerful_test}, and we then choose the rejection region as
$
\sqrt{  n_1n_0 / N }
\max\left\{
T_1/s_{Y_1}, T_2/s_{Y_2}
\right\}>c_\alpha.
$
$\hfill \square$
\end{example}

The previous examples in this section dealt with cases under the sharp null hypotheses that the individual outcomes are unaffected by the treatment. More interestingly, Theorems \ref{thm::repeated-sampling}--\ref{cor:fclt_exp_stable} are useful for obtaining the repeated sampling properties of many causal estimators under different designs. We illustrate this angle with Examples \ref{eg::rerand}--\ref{eg:factorial_design} below.

\begin{example}[Rerandomization]\label{eg::rerand} 
Over complete randomization,
the difference-in-means of covariates, 
$
\widehat{\bm{\tau}}_{\bm{X}}  =  \sum_{i=1}^N Z_i \bm{X}_i / n_1- \sum_{i=1}^N (1-Z_i)\bm{X}_i / n_0,
$
has expectation zero.
However,
for a realized randomization, as pointed by \citet{morgan2012rerandomization}, it is very likely that some covariates are not balanced in means between two treatment groups. Therefore, it is reasonable to accept only those randomizations satisfying certain balancing criterion, such as, a certain norm of the difference-in-means of covariates
is less than or equal to a pre-determined threshold. This is rerandomizaton.  
Note that the covariates $\bm{X}$ are ``outcomes'' unaffected by the treatment, with known finite population covariance $\bm{S}_{\bm{X}}^2$. 
Using Theorem \ref{cor:fclt_exp_stable} with potential outcome $\bm{X}_i$ for both treatment conditions and coefficient matrices $\bm{A}_1=\bm{I}_{K\times K}$ and $\bm{A}_0=-\bm{I}_{K\times K}$, 
over complete randomization,
\begin{align*}
\widehat{\bm{\delta}} \equiv 
\left( \frac{N}{n_1n_0}\bm{S}_{\bm{X}}^2 \right)^{-1/2} \widehat{\bm{\tau}}_{\bm{X}}  
\converged \mathcal{N}(0,\bm{I}_{K\times K}),
\end{align*}
and therefore, $\widehat{\bm{\delta}}^\top \widehat{\bm{\delta}} \converged \chi_K^2$. \cite{morgan2012rerandomization} suggested a rerandomization criterion such that $ \widehat{\bm{\delta}}^\top \widehat{\bm{\delta}}  $ is smaller than a threshold, and obtained theoretical results assuming that $\widehat{\bm{\delta}}$ follows an exact Normal distribution. 
More importantly, we need to study the properties of the difference-in-means estimator $\widehat{\tau} =  \sum_{i=1}^N Z_i Y_i / n_1 -  \sum_{i=1}^N (1-Z_i) Y_i / n_0$ under rerandomization, which is conceptually the same as the conditional distribution of $\widehat{\tau}$ given that $\widehat{\bm{\delta}}$ satisfies the rerandomization criterion over complete randomization. 
Therefore, it is crucial to obtain the joint asymptotic distribution of
$
(
\widehat{\tau}, \widehat{\bm{\tau}}_{\bm{X}} 
)
$
over complete randomization. According to Theorem \ref{cor:fclt_exp_stable} with potential outcomes $(Y(z), \bm{X}_i)$ and coefficient matrices $\bm{A}_1=\bm{I}_{(K+1)\times (K+1)}$ and $\bm{A}_0=-\bm{I}_{(K+1)\times (K+1)}$,
the joint distribution of
$
(
\widehat{\tau}, \widehat{\bm{\tau}}_{\bm{X}} 
)
$ is asymptotically Normal if \eqref{eq:max_condi_eg} holds, a result utilized by 
\citet*{asymrerand2106} to prove asymptotic properties of rerandomization. \citet{cochran1965planning, cochran1970performance} discussed a related problem in observational studies. 
$\hfill \square$
\end{example}

\begin{example}[Regression adjustment in completely randomized experiments]\label{eg:regression_adjust}
Although the simple difference-in-means estimator of the average causal effect is unbiased, carefully utilizing the covariates often improves the efficiency \citep{freedman2008regression_a, freedman2008regression_b,lin2013}.
A class of regression adjustment estimators for the average causal effects is
\begin{align}\label{eq:regress_adj_estimator}
\widehat{\tau}(\bm{\beta}_1, \bm{\beta}_0)=
\frac{1}{n_1}\sum_{i=1}^N Z_i\left(
Y_i - {\bm{\beta}}_1^\top \bm{X}_i
\right) - 
\frac{1}{n_0}\sum_{i=1}^N (1-Z_i)\left(
Y_i - {\bm{\beta}}_0^\top \bm{X}_i
\right),
\end{align}
where $\bm{\beta}_1$ and $\bm{\beta}_0$ are any vectors 
that do not depend on the treatment indicators but can implicitly depend on $N$ with limiting values as $N\rightarrow\infty$.
When $\bm{\beta}_1=\bm{\beta}_0=\bm{0}$, it reduces to the simple difference-in-means estimator.
According to \eqref{eq:regress_adj_estimator}, $\widehat{\tau}(\bm{\beta}_1, \bm{\beta}_0)$ is essentially the difference-in-means estimator with ``adjusted'' treatment potential outcome $Y_i(1) - {\bm{\beta}}_1^\top  \bm{X}_i$ and ``adjusted'' control potential outcome $Y_i(0) - {\bm{\beta}}_0^\top \bm{X}_i$ for unit $i$. Note that the average causal effect based on the adjusted potential outcomes remains the same as the average causal effect based on the original potential outcomes $\tau .$ Therefore, the classical result \citep{Neyman:1923} guarantees that over complete randomization,
$\widehat{\tau}(\bm{\beta}_1, \bm{\beta}_0)$ is unbiased for $\tau$, and the CLT in Theorem \ref{cor:fclt_exp_stable} 
with potential outcomes $Y_i(z)-\bm{\beta}_z^\top \bm{X}_i$ and coefficients $A_1=1$ and $A_0=-1$ guarantees its asymptotic Normality. We can construct a conservative large-sample confidence interval based on the ``adjusted'' observed outcome data $ Y_i - {\bm{\beta}}_1^\top\bm{X}_i$ for treated units with $Z_i=1$ and $Y_i -  {\bm{\beta}}_0^\top \bm{X}_i$ for control units with $Z_i=0$.

In practice, how to choose $\bm{\beta}_1$ and $\bm{\beta}_0$?
Let $\widetilde{\bm{\beta}}_z$ be the finite population least squares coefficient of $Y_i(z)$ on $\bm{X}_i$ for units $i=1,2,\ldots,N$ \citep[cf.][]{cochran1977}. 
We show in the Supplementary Material that the sampling variance of any regression adjustment estimator has the following decomposition:
\begin{align*}
\Var\{
\widehat{\tau}({{\bm{\beta}}_1, {\bm{\beta}}_0})
\} =  \Var\{
\widehat{\tau}({\widetilde{\bm{\beta}}_1, \widetilde{\bm{\beta}}_0})
\}  +\Var\{
\widehat{\tau}({{\bm{\beta}}_1, {\bm{\beta}}_0}) - \widehat{\tau}({\widetilde{\bm{\beta}}_1, \widetilde{\bm{\beta}}_0})
\},
\end{align*}
demonstrating that $\widehat{\tau}({\widetilde{\bm{\beta}}_1, \widetilde{\bm{\beta}}_0})$ is optimal in the sense of having the smallest sampling variance among all regression adjustment estimators defined in \eqref{eq:regress_adj_estimator}.
However, because $\widetilde{\bm{\beta}}_1$ and $\widetilde{\bm{\beta}}_0$ are both unknown, in practice we instead use the sample least squares coefficient of $Y_i$ on $\bm{X}_i$ for units in treatment arm $z$, $\widehat{\bm{\beta}}_z$, to replace $\widetilde{\bm{\beta}}_z.$
In the Supplementary Material, we show that the two regression adjustment estimators with true and 
estimated least squares coefficients, 
$ \widehat{\tau}(
\widetilde{\bm{\beta}}_1, \widetilde{\bm{\beta}}_0)$ and
$ \widehat{\tau}(
\widehat{\bm{\beta}}_1, \widehat{\bm{\beta}}_0)$,  have the same asymptotic Normal distribution, and the difference between them is of order $o_p(1/\sqrt{N}),$ 
without requiring any further regularity conditions beyond \eqref{eq:max_condi_eg}. Furthermore, we show that treating $\widehat{\bm{\beta}}_1$ and $  \widehat{\bm{\beta}}_0$ as if they were pretreatment vectors does not affect the asymptotic coverage rate of the confidence interval based on the Normal approximation.

Recently, \citet{lin2013} established a connection between regression estimator \eqref{eq:regress_adj_estimator} and linear regression with full treatment-covariate interactions using the Huber--White variance estimator. We supplement his result with ``optimality'' and a rigorous proof of the asymptotic Normality. \citet{samii2012equivalencies} focused on comparison of regression-based and randomization-based standard errors, 
\citet{abadie2014finite} gave a proof of the asymptotic Normality of the regression estimator under independent assignment indicators, 
and \citet{rosenbaum2002covariance} discussed alternative forms of covariate adjustment in randomized experiments.
$\hfill \square$
\end{example}

\begin{example}[Cluster-randomized experiments]\label{eg::cluster}
We consider a cluster-randomized experiment, where 
individuals within the same cluster receive the same treatment condition. Assume there are $M$ clusters, and among them, $m_1$ clusters  receive treatment and $m_0$ clusters  receive control.
Let $\widetilde{Y}_j(z)$ be the total potential outcome of units in cluster $j$ under treatment arm $z$,   
$\widetilde{\bm{X}}_j$ a $K$ dimensional cluster-level covariate (including the total number of units and aggregate covariates in each cluster), and $\overline{\widetilde{\bm{X}}}$ be the finite population average of cluster-level covariates. 
To faciliate the discussion, we center the cluster-level covariates at zero ($\overline{\widetilde{\bm{X}}}=\bm{0}$).
Let $N$ be the total number of units, and $\tau = \sum_{j=1}^M \{ \widetilde{Y}_j(1) - \widetilde{Y}_j(0) \}/N$ be the average causal effect over all units.
For cluster $j$, let $\widetilde{Z}_j$ be the treatment assignment indicator, and $\widetilde{Y}_j=\widetilde{Z}_j\widetilde{Y}_j(1)+(1-\widetilde{Z}_j)\widetilde{Y}_j(0)$ be the observed outcome. 
We consider the following class of adjusted estimators for
the average causal effect $\tau$:
\begin{align}\label{eq:reg_adj_cluster}
\widehat{\Delta}(\bm{\gamma}_1, \bm{\gamma}_0) & = \frac{M}{N}\left[
\frac{1}{m_1}
\sum_{j=1}^M \widetilde{Z}_j \left( \widetilde{Y}_j -
\bm{\gamma}_1^\top
\widetilde{\bm{X}}_j 
\right)
- 
 \frac{1}{m_0}\sum_{j=1}^M (1-\widetilde{Z}_j)
 \left(\widetilde{Y}_j- 
\bm{\gamma}_0^\top 
\widetilde{\bm{X}}_j \right)
\right],
\end{align}
where
$\bm{\gamma}_z$'s are any vectors that do not depend on the treatment indicators but can depend implicitly
on $M$ with limiting values as $M \rightarrow\infty$.
When $\bm{\gamma}_1=\bm{\gamma}_0=\bm{0},$ \eqref{eq:reg_adj_cluster} reduces to the simple difference-in-means estimator.
 \citet{Middleton2015cluster} showed that 
$\widehat{\Delta}(\bm{\gamma}_1, \bm{\gamma}_0)$ is unbiased for $\tau$ when $\bm{\gamma}_1=\bm{\gamma}_0$ and they are both predetermined constant vectors. 
Note that if we view each cluster as an experimental unit in a completely randomized experiment, $\widehat{\Delta}(\bm{\gamma}_1, \bm{\gamma}_0)$ is essentially the regression adjustment estimator discussed in Example \ref{eg:regression_adjust}, up to a scale constant $M/N.$ Therefore, all results of Example \ref{eg:regression_adjust} apply here. For instance, we can choose the ``optimal'' adjustment coefficients as $\widehat{\bm{\gamma}}_z$, the sample least squares coefficient of $\widetilde{Y}_j$ on $\widetilde{\bm{X}}_j$ for units in treatment arm $z=1,0$. Due to the similarity to Example \ref{eg:regression_adjust}, we relegate the details about the asymptotic distribution and confidence interval construction to the Supplementary Material. 
$\hfill \square$
\end{example}

\section{Application to factorial experiments}\label{sec::factorial}

Our final example is about unbalanced $2^K$ factorial designs with multiple treatment factors and causal effects, extending \citet{dasgupta2014causal}'s discussion of finite sample properties of balanced $2^K$ factorial designs.

\begin{example}\label{eg:factorial_design}
Consider a factorial design with $K$ factors, where each factor has two levels $+1$ and $-1$, and in total there are $Q=2^K$ treatment combinations. For each treatment combination $1\leq q\leq Q$, let $\bm{z}(q)=(z_1(q), z_2(q),\ldots, z_K(q))$ be the levels of the $K$ factors,
and $n_q$ be the number of units. Let $N=\sum_{q=1}^Q n_q$ be the total number of units, $Y_i(q)$ the potential outcome of unit $i$ and $\bar{Y}(q)=\sum_{i=1}^N Y_i(q)/N$ the average potential outcome under treatment combination $q$. Let 
$\bm{Y}_i(1\text{:}Q)=\left(
Y_i(1), Y_i(2), \ldots, Y_i(Q)
\right)$
be the $Q$ dimensional row vector consisting of 
 unit $i$'s potential outcomes under all treatment combinations,
and $
\bar{\bm{Y}}(1\text{:}Q) = (\bar{Y}(1), \bar{Y}(2), \ldots, \bar{Y}(Q))$ be the row vector consisting of all average potential outcomes. 
Following the notation in \citet{dasgupta2014causal}, each factorial effect can be characterized by a column vector with half of its elements being $+1$ and the other half being $-1$. For example, the average main effect of factor $1$ is
\begin{align*}
\tau_1 & = \frac{1}{2^{K-1}}\sum_{q=1}^Q 1\{z_{1}(q)=1\}\cdot \bar{Y}(q) - \frac{1}{2^{K-1}}\sum_{q=1}^Q 1\{z_{1}(q)=-1\}\cdot \bar{Y}(q) \\
& = \frac{1}{2^{K-1}}\sum_{q=1}^Q z_{1}(q)\bar{Y}(q) = \frac{1}{2^{K-1}}\bar{\bm{Y}}(1\text{:}Q)\bm{g}_1,
\end{align*}
where $\bm{g}_1=(z_{1}(1), z_{1}(2)\ldots, z_{1}(Q))^\top$ characterizes the first factorial effect. In general, let
$\bm{g}_k=(g_{k1}, \ldots, g_{kQ})^\top \in \{+1, -1\}^Q$ be the vector generating the $k$th factorial effect, 
and
$
\tau_k =  2^{-(K-1)}\bar{\bm{Y}}(1\text{:}Q)\bm{g}_k  
$
be the $k$th average factorial effect. 

For each unit $i$, let $L_i$ be the treatment assignment indicator  ($L_i=q$ if unit $i$ receives treatment combination $q$), and $Y_i= Y_i(L_i)$ be the observed outcome. Let $\widehat{\bar{Y}}(q)=\sum_{i=1}^N 1\{L_i=q\}Y_i/n_q$ be the average observed outcome under treatment combination $q$, and $\widehat{\bar{\bm{Y}}}(1\text{:}Q) = (\widehat{\bar{Y}}(1), \widehat{\bar{Y}}(2), \ldots, \widehat{\bar{Y}}(Q))$ the row vector consisting of all average observed outcomes. 
An unbiased estimator for $\tau_k$ is
$$
\widehat{\tau}_k = 2^{-(K-1)}\widehat{\bar{\bm{Y}}}(1\text{:}Q)\bm{g}_k = 2^{-(K-1)} \sum_{q=1}^Q g_{kq}\widehat{\bar{Y}}(q) \quad (1\leq k \leq Q-1).
$$

We first consider the joint asymptotic distribution of the $\widehat{\tau}_k$'s under the sharp null hypothesis that $Y_i(1)=\cdots =Y_i(Q)$ for all units $i$. Let $m_N$ be the maximum squared distance of the  $Y_i$'s from the average, and $v_N$ be the finite population variance. 
From Theorem \ref{thm::repeated-sampling}, under the sharp null hypothesis, for any $1\leq k,m\leq Q-1,$ the variance of $\widehat{{\tau}}_k$  is 
$
\Var_0( \widehat{\tau}_{k} )  =2^{-2(K-1)}
\sum_{q=1}^Q n_q^{-1}v_N,
$ 
and the covariance between $\widehat{{\tau}}_k$ and $\widehat{{\tau}}_m$ is $\Cov_0(\widehat{\tau}_k, \widehat{\tau}_m) = 2^{-2(K-1)}\sum_{q=1}^{Q} n_q^{-1}g_{kq}g_{mq} v_N$. 
According to Corollary \ref{cor:add_causal_effects}, as $N\rightarrow \infty$, if for each $1\leq q\leq Q$, $m_N/(n_qv_N)$ converges to zero, and $n_q/N$ has a positive limit, then all the $\widehat{\tau}_k$'s are jointly Normal asymptotically. Without the sharp null hypothesis, the asymptotic Normality of the $\widehat{{\tau}}_k$'s over repeated sampling can be similarly established, and it is straightforward to extend it to regression adjustment in factorial experiments \citep{lu2016covariate, lu2016randomization}.

If we consider only the marginal distribution of $\widehat{\tau}_k$ under the sharp null hypothesis, then the regularity condition for asymptotic Normality can be weakened. 
In practice, it is likely that both the total number of units $N$ and the total number of treatment combinations $Q=2^K$ are large, but the number of units in each treatment combination is moderate. 
Instead of assuming that the total number of treatments is fixed and does not change as $N$ increases, we allow the number of total treatment combinations $Q$ to increase as $N$ becomes larger. 
Under the sharp null hypothesis, if the design is balanced with $n_q = N/Q,$
then the $\widehat{\tau}_k$'s have exactly the same distribution by symmetry, although their realized numerical values can be different. 
Without loss of generality, we consider only $\widehat{\tau}_1$, which is essentially the difference-in-means of a random half versus the remaining half of the observed values of the outcomes.
Its asymptotic Normality follows directly from Theorem \ref{thm:CLThajek}, if $\{Y_i:i=1,\ldots,N\}$ satisfies Condition \eqref{eq:condi_CLThajek}. The theory for unbalanced designs appears to be more technical, and we defer the discussion to the Supplementary Material.
$\hfill \square$
\end{example}

\section{Discussion}\label{sec:discuss}
We have established general forms of finite population CLTs, which were frequently invoked implicitly in randomization-based causal inference. We use them to study asymptotic properties of randomization-based causal inference in completely randomized experiments, cluster randomized experiments, and factorial experiments. For stratified experiments, each stratum is essentially a completely randomized experiment \citep{kempthorne1952design,kempthorne1994design}. Therefore, if the number of strata is small but the sample size is large within each stratum, then we can apply the finite population CLTs to each stratum, and average over the strata to obtain the finite population CLTs for the causal estimators; if the number of independent strata is large, then it suffices to use the classical CLTs for independent variables. Matched-pair experiments are special cases of stratified experiments with two units within each stratum, and the CLT requires a large number of pairs \citep{rosenbaum2002observational, imai2008variance}. Our finite population CLTs may also be useful for other experimental designs--a prospect that needs further investigation and analyses \citep{kempthorne1952design,kempthorne1994design}.


\bibliographystyle{plainnat}
\bibliography{causal}

\newpage
\setcounter{page}{1}
\begin{center}
\bf \huge 
Supplementary Material
\end{center}

\bigskip

\setcounter{equation}{0}
\setcounter{section}{0}
\setcounter{figure}{0}
\setcounter{example}{0}
\setcounter{proposition}{0}
\setcounter{corollary}{0}
\setcounter{theorem}{0}
\setcounter{table}{0}

\renewcommand {\theproposition} {A.\arabic{proposition}}
\renewcommand {\theexample} {A.\arabic{example}}
\renewcommand {\thefigure} {A.\arabic{figure}}
\renewcommand {\thetable} {A.\arabic{table}}
\renewcommand {\theequation} {A.\arabic{equation}}
\renewcommand {\thelemma} {A.\arabic{lemma}}
\renewcommand {\thesection} {A.\arabic{section}}
\renewcommand {\thetheorem} {A.\arabic{theorem}}
\renewcommand {\thecorollary} {A.\arabic{corollary}}

Appendix \ref{app:theorem_coro} provides the proofs for the theorems and corollaries, and Appendix \ref{app:more_eg} provides details about some examples discussed in the main text.

\section{Comments and Proofs of the Theorems and Corollaries}\label{app:theorem_coro}

\begin{proof}[{\bf Comments on Theorem \ref{thm:CLThajek}}]
First, we show that Condition \eqref{eq:condi_CLThajek} is equivalent to the regulation condition in \citet[][Appendix 4]{lehmann2006nonparametrics}:
\begin{align}\label{eq:condition_in_lehmann}
\frac{1}{N}\cdot\frac{  m_N }{ v_N } \max\left(  \frac{N-n}{n}, \frac{n}{N-n} \right) \rightarrow 0.
\end{align}
They are equivalent because up to some fixed constants, the left-hand side of \eqref{eq:condition_in_lehmann} can be bounded from below and above by the left-hand side of Condition \eqref{eq:condi_CLThajek}:
\begin{align*}
\frac{1}{N}\max\left(  \frac{N-n}{n}, \frac{n}{N-n} \right)
\begin{cases}
\leq \frac{1}{N} \max\left(  \frac{N}{n}, \frac{N}{N-n} \right) = \max\left(  \frac{1}{n}, \frac{1}{N-n} \right) = \frac{1}{\min(n,N-n)},\\
\geq \frac{1}{N}\cdot\frac{1}{2}\left\{\max\left(  \frac{N-n}{n}, \frac{n}{N-n} \right) + 1 \right\} = \frac{1}{2} \max\left(  \frac{1}{n}, \frac{1}{N-n} \right)=\frac{1}{2}\frac{1}{\min(n,N-n)}. 
\end{cases}
\end{align*}

Second, we show that if the $y_{Ni}$'s are i.i.d. draws 
from a  super population with $2+\varepsilon$ $(\varepsilon>0)$ moment and nonzero variance, then
$m_N/N\rightarrow 0$ holds with probability one.
Let the random variable $Y$ denote the super population distribution. 
Note that
$$
(y_{Ni}-\bar{y}_N)^2\leq 2y_{Ni}^2 + 2\bar{y}_N^2, \quad \max_{1\leq i \leq N} y_{Ni}^2 = \left(\max_{1\leq i\leq N}  |  y_{Ni} |^{2+\varepsilon} \right)^{2/(2+\varepsilon)} \leq  \left(\sum_{i=1}^N   |  y_{Ni} |^{2+\varepsilon} \right) ^{2/(2+\varepsilon)}. 
$$ 
Therefore,
\begin{align*}
\frac{m_N}{N}
 \leq \frac{2}{N}\bar{y}_N^2  + \frac{2}{N}\max_{1\leq i\leq N} y_{Ni}^2
  \leq  \frac{2}{N}\bar{y}_N^2 + \frac{2}{N} \left(
\sum_{i=1}^N |y_{Ni}|^{2+\varepsilon}
\right)^{2/(2+\varepsilon)} = 
 \frac{2}{N}\bar{y}_N^2 + 
\frac{2}{N^{\varepsilon/(2+\varepsilon)}}
\left(
\frac{1}{N}
\sum_{i=1}^N |y_{Ni}|^{2+\varepsilon}
\right)^{2/(2+\varepsilon)}.
\end{align*}
By the law of large numbers, 
$\bar{y}_N^2\rightarrow \{E(Y)\}^2$ and
$N^{-1}
\sum_{i=1}^N |y_{Ni}|^{2+\varepsilon}\rightarrow E(|Y|^{2+\varepsilon})$ almost surely, and therefore $m_N/N\rightarrow 0$ almost surely.

Third, we consider a sufficient and necessary condition for the asymptotic Normality of $\ybars$. 
Let $\{\delta_{Ni}\}_{i=1}^N$ be the standardized finite population of size $N$ defined as $\delta_{Ni}=(y_{Ni}-\bar{y}_N)/\sqrt{v_N}$. 
According to \citet[][Theorem 3.1]{hajek1960limiting}, assume that $n\rightarrow\infty$ and  $N-n\rightarrow\infty$.
Then $\ybars$ is asymptotically Normal, if and only if,  for any $\varepsilon>0$,
\begin{align}\label{eq:hajek_suff_nece}
\frac{1}{N-1}\sum_{i: |\delta_{Ni}|>\varepsilon\sqrt{n(N-n)/N}} \delta_{Ni}^2 \rightarrow 0. 
\end{align}
This condition is similar to a condition of the classical Lindeberg--Feller CLT for independent variables. 
The regularity conditions in  \citet[][Theorem 3.1]{hajek1960limiting} is weaker than Condition \eqref{eq:condi_CLThajek}, because  
(a) as commented in the main text, 
Condition \eqref{eq:condi_CLThajek} implies $n\rightarrow\infty$ and  $N-n\rightarrow\infty$; 
(b) Condition \eqref{eq:condi_CLThajek} implies
\begin{align*}
\max_{1\leq i\leq N}
\frac{\delta_{Ni}^2}{n(N-n)/N} = \frac{N}{n(N-n)}\max_{1\leq i\leq N}\frac{(y_{Ni}-\bar{y}_N)^2}{v_N} = \left(
\frac{1}{n} + \frac{1}{N-n}
\right)\frac{m_N}{v_N}\leq \frac{2}{\min(n, N-n)}\cdot \frac{m_N}{v_N}\rightarrow 0. 
\end{align*}
Therefore, for any $\varepsilon>0$, there exists an $N_\varepsilon$ such that when $N\geq N_\varepsilon$, $\max_{1\leq i\leq N}|\delta_{Ni}|\leq \varepsilon \sqrt{n(N-n)/N}$, and thus  $(N-1)^{-1}\sum_{i: |\delta_{Ni}|>\varepsilon\sqrt{n(N-n)/N}} \delta_{Ni}^2 = 0.$ 
This then guarantees that Condition \eqref{eq:hajek_suff_nece} holds for any $\varepsilon>0$.
\end{proof}

\begin{proof}[{\bf Proof of Proposition \ref{prop:v_N_consist}}]
From the definition of $\widehat{v}_N$,
\begin{align}\label{eq:ratio_v_proof}
\frac{\widehat{v}_N}{v_N} = 
\frac{1}{v_N}\frac{n}{n-1}\left\{
\frac{1}{n}\sum_{i=1}^N Z_i (y_{Ni}-\bar{y}_N)^2 - (\ybars-\bar{y}_N)^2
\right\}.
\end{align}
 Let $w_N$ be the finite population variance of $\{(y_{Ni}-\bar{y}_N)^2: 1\leq i \leq N\}$. 
The Markov inequality implies $\ybars-\bar{y}_N   = O_p\left(
\sqrt{\Var(\ybars)}
\right) = O_p\left(\sqrt{v_N/n}\right),$ and similarly,  
\begin{align*}
\frac{1}{n}\sum_{i=1}^N Z_i (y_{Ni}-\bar{y}_N)^2 - \frac{N-1}{N}v_N 
 & =  O_p\left( \sqrt{\frac{w_N}{n}} \right).
\end{align*}
Thus, replacing these two terms in \eqref{eq:ratio_v_proof}, we have 
\begin{align*}
\frac{\widehat{v}_N}{v_N} & = 
\frac{1}{v_N}\frac{n}{n-1}\left\{
\frac{N-1}{N}v_N + O_p\left( \sqrt{\frac{w_N}{n}} \right) - O_p\left( \frac{v_N}{n} \right)
\right\} = 1 + O_p\left( \sqrt{\frac{w_N}{nv_N^2}} \right) +O_p\left(
\frac{1}{n}
\right).
\end{align*}
Note that $w_N$ can be bounded by the fourth central moment of the $y_{Ni}$'s, which can be further bounded by the product of $m_N$ and $v_N$: 
\begin{align*}
w_N \leq \frac{1}{N-1}\sum_{i=1}^N (y_{Ni}-\bar{y}_N)^4 \leq \frac{1}{N-1}m_N \sum_{i=1}^N (y_{Ni}-\bar{y}_N)^2 = m_Nv_N. 
\end{align*}
This allows us to show that $w_N/(nv_N^2)$ converges to zero under Condition \eqref{eq:condi_CLThajek} in Theorem \ref{thm:CLThajek}:
\begin{align*}
\frac{w_N}{nv_N^2}& \leq \frac{m_Nv_N}{nv_N^2} = 
\frac{1}{n}\frac{m_N}{v_N} \leq \frac{1}{\min(n,N-n)} \cdot \frac{m_N}{v_N}
\rightarrow 0.
\end{align*}
Therefore,  ${\widehat{v}_N}/{v_N}=1+o_p(1)$.
\end{proof}

To prove Theorem \ref{thm:sim_clt}, we need the following lemma of \citet{fraser1956vector}.
\begin{lemma}\label{lemma:fraser}
Let $\left\{C_{Nq}(i,j)\right\}_{i,j=1}^N$ ($1\leq q \leq Q$) be $Q$ finite populations of size $N^2$, $(J_1,\ldots,J_N)$ be a random vector has probability $(N!)^{-1}$ to be any permutation of $\{1,\ldots,N\}$, and $Y_{Nq}=\sum_{i=1}^N C_{Nq}(i,J_i)$ for $1\leq q \leq Q$. For any $1\leq q,r\leq Q$ and $1\leq i,j\leq N$, let $\rho_{N,qr}$ be the correlation between $Y_{Nq}$ and $Y_{Nr}$,
 and
\begin{align*}
D_{Nq}(i,j) = C_{Nq}(i,j) - \frac{1}{N}\sum_{i'=1}^N C_{Nq}(i',j) - \frac{1}{N}\sum_{j'=1}^N C_{Nq}(i,j') + \frac{1}{N^2}\sum_{i'=1}^N \sum_{j'=1}^N C_{Nq}(i',j')
\end{align*}
be the standardized finite population elements. 
As $N\rightarrow\infty$, 
if for all $q,r=1,\ldots,Q$, (a) the correlation $\rho_{N,qr}$ has a limiting value $\rho_{qr}$, and (b)
\begin{align*}
\frac{
\max_{1\leq i,j\leq N}D_{Nq}^2(i,j)
}{
N^{-1}\sum_{i=1}^N\sum_{j=1}^N D_{Nq}^2(i,j)
}\rightarrow 0,
\end{align*}
then
\begin{align*}
\left(
\frac{Y_{N1}-E(Y_{N1})}{\sqrt{\Var(Y_{N1})}}, \ldots, 
\frac{Y_{NQ}-E(Y_{NQ})}{\sqrt{\Var(Y_{NQ})}}
\right)^\top \converged \mathcal{N}(\bm{0}, \bm{V}),
\end{align*}
where $\bm{V}$ is a $Q\times Q$ correlation matrix with the $(q,r)$-th entry $\rho_{qr}$ and diagonal elements 1.
\end{lemma}

\begin{proof}[{\bf Proof of Theorem \ref{thm:sim_clt}}]
To employ Lemma \ref{lemma:fraser},
we first construct $Q$ finite populations of size $N^2$ as follows: 
for any $1\leq i \leq N$ and $1\leq q\leq Q$, define
\begin{align*}
C_{Nq}(i,j) = 
\begin{cases}
y_{Ni}, & \sum_{r=1}^{q-1}n_r < j \leq \sum_{r=1}^{q}n_r, \\
0, & \text{otherwise}.
\end{cases}
\end{align*}
Let $(J_1,\ldots,J_N)$ be a random vector has probability $(N!)^{-1}$ to be any permutation of $\{1,\ldots,N\}$, and $L_i=q$ if  $\sum_{r=1}^{q-1}n_r < J_i \leq \sum_{r=1}^{q}n_r$. Then $(L_1,\ldots, L_N)$ is a random partition of the $N$ units into $Q$ groups of size $(n_1,\ldots,n_Q)$, and
\begin{align*}
Y_{Nq} = \sum_{i=1}^N C_{Nq}(i,J_i) = \sum_{i:L_i=q}y_{Ni} = n_q \bar{y}_{\text{S}q}
 \quad (1\leq q\leq Q).
\end{align*}
For each $1\leq q\leq Q$,
\begin{align*}
D_{Nq}(i,j) 
& = 
\begin{cases}
\left( 1- \frac{n_q}{N}\right)(y_{Ni}-\bar{y}_N), & \sum_{r=1}^{q-1}n_r < j \leq \sum_{r=1}^{q}n_r, \\
- \frac{n_q}{N}(y_{Ni}-\bar{y}_N), & \text{otherwise}.
\end{cases}
\end{align*}
We can verify that Condition \eqref{eq:regu_sim_clt} of Theorem \ref{thm:sim_clt} implies Condition (b) of Lemma \ref{lemma:fraser}, because for each $1\leq q\leq Q$,
\begin{align*}
\frac{
\max_{1\leq i,j\leq N}D_{Nq}^2(i,j)
}{
N^{-1}\sum_{i=1}^N\sum_{j=1}^N D_{Nq}^2(i,j)
} & = \frac{m_{N}\max\left\{\left(1-\frac{n_q}{N}\right)^2, \left(\frac{n_q}{N}\right)^2\right\}}{
N^{-1}\sum_{i=1}^N(y_{Ni}-\bar{y}_N)^2\left\{
n_q\left( 1-\frac{n_q}{N} \right)^2 + (N-n_q)\left( \frac{n_q}{N} \right)^2
\right\}
}\\
& = \frac{N}{N-1} \frac{1}{N}\frac{m_N}{v_N} \max \left(\frac{N-n_q}{n_q}, \frac{n_q}{N-n_q}\right)\leq \frac{N}{N-1} \frac{1}{N}\frac{m_N}{v_N} \max \left(\frac{N}{n_q}, \frac{N}{N-n_q}\right)\\
& = \frac{N}{N-1} \frac{1}{\min(n_q, N-n_q)} \frac{m_N}{v_N} \leq \frac{N}{N-1} \frac{1}{\min_{1\leq r\leq Q}n_r} \frac{m_N}{v_N}
\rightarrow 0.
\end{align*}
Note that Condition (a) of Lemma \ref{lemma:fraser} is equivalent to the convergence of $\Var(\bm{t}_N)$. Therefore, 
according to Lemma \ref{lemma:fraser}, 
Theorem \ref{thm:sim_clt} holds.
\end{proof}

\begin{proof}[{\bf Comments on Theorem \ref{thm:sim_clt}}]
We show that Condition \eqref{eq:regu_sim_clt} is equivalent to the regularity condition in \citet[][Theorem 19, page 393]{lehmann2006nonparametrics}:
\begin{align}\label{eq:lehmann_rand_part}
\frac{1}{N}
\frac{ 
m_N
}{ v_N}
\max_{1\leq q\leq Q}\left(
\frac{N-n_q}{n_q}, \frac{n_q}{N-n_q}
\right) \rightarrow 0.
\end{align} 
They are equivalent because up to some fixed constants, the left-hand side of \eqref{eq:lehmann_rand_part} can be bounded from below and above by the left-hand side of Condition \eqref{eq:regu_sim_clt}:
\begin{align*}
\max_{1\leq q\leq Q}\left(
\frac{N-n_q}{n_q}, \frac{n_q}{N-n_q}
\right) 
\begin{cases}
\leq 
\max_{1\leq q\leq Q}\left(
\frac{N}{n_q}, \frac{N}{N-n_q}
\right) = N \max_{1\leq q\leq Q}\left(
\frac{1}{n_q}, \frac{1}{N-n_q}
\right) = \frac{N}{\min_{1\leq q\leq Q} n_q}, 
\\
\geq \frac{1}{2}\left\{
\max\limits_{1\leq q\leq Q}\left(
\frac{N-n_q}{n_q}, \frac{n_q}{N-n_q}
\right) + 1\right\} = \frac{N}{2}\max\limits_{1\leq q\leq Q}\left(
\frac{1}{n_q}, \frac{1}{N-n_q}
\right) = \frac{N/2}{\min_{1\leq q\leq Q} n_q}.
\end{cases}
\end{align*}
\end{proof}

\begin{proof}[{\bf Proof of Corollary \ref{cor::rank}}]
Let $\Pi_{N}=\{y_{Ni}=i: 1\leq i\leq N\}$ be a finite population, $\bar{y}_{\text{S}q} = \sum_{i:{L}_i=q}y_{Ni}/n_q$ be the sample averge of group $q$, $v_N=N(N+1)/12$ be the finite population variance of $y_{Ni}$, and
$m_N=(N-1)^2/4$ be the maximum squared distance of $y_{Ni}$'s from  $\bar{y}_N$.
Under the sharp null hypothesis that
$
Y_i(1) = Y_i(2) = \cdots =Y_i(Q)
$
for all units $i$, $(\bar{R}_{(1)}, \ldots, \bar{R}_{(Q)})$ has the same distribution as $(\bar{y}_{\text{S}1}, \ldots, \bar{y}_{\text{S}Q})$, and therefore $(\widetilde{R}_{(1)}, \ldots, \widetilde{R}_{(Q)})$ has the same distribution as $\bm{t}_N$ defined in Theorem \ref{thm:sim_clt}.
According to the regularity conditions in Corollary \ref{cor::rank}, $\Var(\bm{t}_N)$ has a limiting value $\bm{V}_R$, and Condition \eqref{eq:regu_sim_clt} holds:  
\begin{align*}
\frac{1}{\min_{1\leq q\leq Q}n_q}
\frac{ 
m_N
}{ v_N}
 = \frac{1}{\min_{1\leq q\leq Q}n_q} \frac{3(N-1)^2}{N(N+1)}\rightarrow 0.
\end{align*}
Using Theorem \ref{thm:sim_clt}, Corollary \ref{cor::rank} holds.
\end{proof}

To calculate the first two moments of $\widehat{\bm{\tau}}(\bm{A})$, we need first to calculate the mean and covariance structure of group indicators $(L_1,\ldots, L_N)$. We summarize the results in the following lemma, the proof of which is straightforward and thus omitted. 

\begin{lemma}\label{lemma:cov_Ls}
Under random partition with group sizes $(n_1, \ldots, n_Q)$, 
for any $1\leq q, r\leq Q$ and units $i$ and $j$, the probability that unit $i$ is partitioned into group $q$ is $P(L_i=q)=n_q/N$, and 
the covariance between the group indicators of units $i$ and $j$  is
\begin{eqnarray*}
\Cov\left(1\{L_i=q\}, 1\{L_j=r\}\right) & = & 
\begin{cases}
\frac{n_q(N-n_q)}{N^2}, & \text{if } i=j \text{ and } q=r, \\
- \frac{n_qn_r}{N^2},  & \text{if } i=j \text{ and } q\neq r, \\
-\frac{n_q(N-n_q)}{N^2(N-1)}, & \text{if } i\neq j \text{ and } q=r, \\
\frac{n_q n_r}{N^2(N-1)}, & \text{if } i\neq j \text{ and } q\neq r.
\end{cases}
\end{eqnarray*}
\end{lemma}

\begin{proof}[{\bf Proof of Theorem \ref{thm::repeated-sampling}}]
By definition,
\begin{align*}
\widehat{\bm{\tau}}(\bm{A}) & = \sum_{q=1}^Q \bm{A}_i \widehat{\bar{\bm{Y}}}(q)= \sum_{q=1}^Q \bm{A}_q\cdot 
\frac{1}{n_q}
\sum_{i:L_i=q}\bm{Y}_i(q) = \sum_{q=1}^Q 
\frac{1}{n_q}
\sum_{i=1}^N 1\{L_i=q\}\bm{A}_q\bm{Y}_i(q),
\end{align*}
where the only random components are $(L_1,\ldots,L_N)$.
Therefore, by linearity and Lemma \ref{lemma:cov_Ls}, 
\begin{align*}
E\left\{ \widehat{\bm{\tau}}(\bm{A}) \right\} & = \sum_{q=1}^Q 
\frac{1}{n_q}
\sum_{i=1}^N  P(L_i=q)\bm{A}_q\bm{Y}_i(q)
= \sum_{q=1}^Q  
\frac{1}{n_q}
\sum_{i=1}^N \frac{n_q}{N} \cdot 
\bm{A}_q\bm{Y}_i(q)
= \sum_{q=1}^Q\bm{A}_q\bar{\bm{Y}}(q) = \bm{\tau}(\bm{A}).
\end{align*}
For each $1\leq q\leq Q$, because
$\sum_{i=1}^N 1\{L_i=q\}\bm{A}_q\bm{Y}_i(q)/n_q$ is the sample average of the values of the $\bm{A}_q\bm{Y}_i(q)$'s in a simple random sample of size $n_q$, according to \citet{cochran1977} for the scalar case
\begin{align}\label{eq:var_sub_ave}
\Cov\left( 
\frac{1}{n_q}
\sum_{i=1}^N 1\{L_i=q\}\bm{A}_q\bm{Y}_i(q) \right)
& =  \left( \frac{1}{n_q} - \frac{1}{N} \right)\bm{A}_q \bm{S}_q^2 \bm{A}_q^\top.
\end{align}
For any $1\leq i\leq N$ and $1\leq q\leq Q$, let $\widecheck{\bm{Y}}_i(q) = \bm{Y}_i(q)-\bar{\bm{Y}}(q)$ be the centered potential outcome. 
For any $1\leq q\neq r\leq Q$,
\begin{eqnarray*}
& & \Cov\left( 
\frac{1}{n_q}
\sum_{i=1}^N 1\{L_i=q\}\bm{A}_q\bm{Y}_i(q), 
\frac{1}{n_r}
\sum_{j=1}^N 1\{L_j=r\}\bm{A}_r\bm{Y}_j(r) \right)\\
& = & 
\frac{1}{n_qn_r}
\sum_{i=1}^N \Cov\left(1\{L_i=q\}, 1\{L_i=r\} \right) \bm{A}_q \widecheck{\bm{Y}}_i(q)
\widecheck{\bm{Y}}_i(r)^\top \bm{A}_r^\top\\
& & + \frac{1}{n_qn_r}
\sum_{i=1}^N\sum_{j \neq i} \Cov\left(1\{L_i=q\}, 1\{L_j=r\} \right) \bm{A}_q \widecheck{\bm{Y}}_i(q)
\widecheck{\bm{Y}}_j(r)^\top \bm{A}_r^\top\\
& = &
\frac{1}{n_qn_r}
\sum_{i=1}^N \left( -\frac{n_qn_r}{N^2} \right) \bm{A}_q \widecheck{\bm{Y}}_i(q)
\widecheck{\bm{Y}}_i(r)^\top \bm{A}_r^\top + 
\frac{1}{n_qn_r}
\sum_{i=1}^N\sum_{j \neq i}
\frac{n_q n_r}{N^2(N-1)}
\bm{A}_q \widecheck{\bm{Y}}_i(q)
\widecheck{\bm{Y}}_j(r)^\top \bm{A}_r^\top,
\end{eqnarray*}
where the last equality follows from Lemma \ref{lemma:cov_Ls}. The above expression can be simplified as: 
\begin{eqnarray}\label{eq:cov_sub_ave}
& &  
-\frac{1}{N^2}
\sum_{i=1}^N \bm{A}_q \widecheck{\bm{Y}}_i(q)
\widecheck{\bm{Y}}_i(r)^\top \bm{A}_r^\top + 
\frac{1}{N^2(N-1)}
\sum_{i=1}^N\sum_{j \neq i}
\bm{A}_q \widecheck{\bm{Y}}_i(q)
\widecheck{\bm{Y}}_j(r)^\top \bm{A}_r^\top \nonumber\\
& = & - \left(\frac{1}{N^2}+\frac{1}{N^2(N-1)}\right)
\sum_{i=1}^N \bm{A}_q \widecheck{\bm{Y}}_i(q)
\widecheck{\bm{Y}}_i(r)^\top \bm{A}_r^\top + 
\frac{1}{N^2(N-1)}
\sum_{i=1}^N\sum_{j=1}^N \bm{A}_q \widecheck{\bm{Y}}_i(q)
\widecheck{\bm{Y}}_j(r)^\top \bm{A}_r^\top \nonumber\\
& = & - \frac{1}{N(N-1)}
\sum_{i=1}^N \bm{A}_q \widecheck{\bm{Y}}_i(q)
\widecheck{\bm{Y}}_i(r)^\top \bm{A}_r^\top + 
\bm{0}
 = -\frac{1}{N}\bm{A}_q \bm{S}_{qr}\bm{A}_r^\top,
\end{eqnarray}
where the last equality follows from the definition of $\bm{S}_{qr}$.
Therefore, the variance of $\widehat{\bm{\tau}}(\bm{A})$ is
\begin{eqnarray*}
\Cov\left\{ \widehat{\bm{\tau}}(\bm{A}) \right\} & = &
\sum_{q=1}^Q \Cov\left( 
\frac{1}{n_q}
\sum_{i=1}^N 1\{L_i=q\}\bm{A}_q\bm{Y}_i(q) \right) \\
& & + \sum_{q\neq r}\Cov\left( 
\frac{1}{n_q}
\sum_{i=1}^N 1\{L_i=q\}\bm{A}_q\bm{Y}_i(q), 
\frac{1}{n_r}\sum_{j=1}^N 1\{L_j=r\}\bm{A}_r\bm{Y}_j(r) \right)\\
& = & \sum_{q=1}^Q \left( \frac{1}{n_q} - \frac{1}{N} \right)\bm{A}_q \bm{S}_q^2 \bm{A}_q^\top-\sum_{q=1}^Q 
\sum_{r\neq q}\frac{1}{N}\bm{A}_q \bm{S}_{qr}\bm{A}_r^\top,
\end{eqnarray*}
following from \eqref{eq:var_sub_ave} and \eqref{eq:cov_sub_ave}. The above covariance formula can be simplified as: 
\begin{align*}
\Cov\left\{ \widehat{\bm{\tau}}(\bm{A}) \right\}  =   \sum_{q=1}^Q \frac{1}{n_q}\bm{A}_q \bm{S}_q^2 \bm{A}_q^\top - \frac{1}{N}\left(
 \sum_{q=1}^Q\bm{A}_q \bm{S}_q^2 \bm{A}_q^\top+
\sum_{q=1}^Q 
\sum_{r\neq q}\bm{A}_q \bm{S}_{qr}\bm{A}_r^\top
\right)
 = \sum_{q=1}^Q \frac{1}{n_q}\bm{A}_q \bm{S}_q^2 \bm{A}_q^\top - \frac{1}{N}\bm{S}^2_{\bm{\tau}(\bm{A})},
\end{align*}
where the last equality follows from the decomposition of the covariance matrix of the individual casual effects: 
\begin{align}\label{eq:S_tau}
\bm{S}^2_{\bm{\tau}(\bm{A})} & = \frac{1}{N-1}
\sum_{i=1}^N \left\{
\bm{\tau}_i(\bm{A})- \bm{\tau}(\bm{A})
\right\}
\left\{
\bm{\tau}_i(\bm{A})- \bm{\tau}(\bm{A})
\right\}^\top  
 = 
\frac{1}{N-1}
\sum_{i=1}^N 
\left( \sum_{q=1}^Q \bm{A}_q 
\widecheck{\bm{Y}}_i(q) \right)
\left(
\sum_{r=1}^Q \bm{A}_r
\widecheck{\bm{Y}}_i(r) \right)^\top \nonumber
\\
& = \sum_{q=1}^Q \bm{A}_q \bm{S}_{q}^2\bm{A}_q^\top+
\sum_{q=1}^Q\sum_{r\neq q} \bm{A}_q \bm{S}_{qr}\bm{A}_r^\top.
\end{align}
\end{proof}

\begin{proof}[{\bf Proof of Theorem \ref{thm:fclt_cre}}]
Recall that $[\cdot]_{(k)}$ is the $k$-th coordinate of a vector. 
To employ Lemma \ref{lemma:fraser}, we first construct 
$K$ finite populations of size $N^2$ as follows:
for any $1\leq i \leq N$ and $1\leq k\leq K$,
\begin{align*}
C_{Nk}(i,j) =
\begin{cases}
\left[ n_1^{-1}\bm{A}_1\bm{Y}_i(1) \right]_{(k)}, & 1\leq j \leq n_1, \\
\left[ n_2^{-1}\bm{A}_2\bm{Y}_i(2) \right]_{(k)}, & 
n_1+1 \leq j\leq n_1+n_2, \\
\quad \quad \quad \vdots & \\
\left[ n_Q^{-1}\bm{A}_Q\bm{Y}_i(Q) \right]_{(k)}, & 
\sum_{r=1}^{Q-1}n_r+1 \leq j\leq N.
\end{cases} 
\end{align*}
Let $(J_1,\ldots,J_N)$ be a random vector has probability $(N!)^{-1}$ to be any permutation of $\{1,\ldots,N\}$, and $L_i=q$ if  $\sum_{r=1}^{q-1}n_k < J_i \leq \sum_{r=1}^{q}n_k$. Then $(L_1,\ldots, L_N)$ is a random partition of the $N$ units into $Q$ groups of size $(n_1,\ldots,n_Q)$, and
\begin{align*}
\left[ \widehat{\bm{\tau}}(\bm{A}) \right]_{(k)} & = \sum_{q=1}^Q
\left[  \sum_{i=1}^N 1\{L_i=q\} n_q^{-1}\bm{A}_q \bm{Y}_i(q)\right]_{(k)} =  \sum_{i=1}^N \left[ n_{L_i}^{-1}\bm{A}_{L_i}\bm{Y}_i(L_i) \right]_{(k)}
=  \sum_{i=1}^N  C_{Nk}(i,J_i).
\end{align*}
For any $1\leq i \leq N$ and $1\leq q\leq Q$, let $\widecheck{\bm{Y}}_i(q)= \bm{Y}_i(q)-\bar{\bm{Y}}(q)$ and $\widecheck{\bm{\tau}}_i(\bm{A})=\bm{\tau}_i(\bm{A})-\bm{\tau}(\bm{A})$ be the centered potential outcome and individual causal effect.
For each $1\leq k\leq K$, 
\begin{align*}
D_{Nk}(i,j) & = C_{Nk}(i,j) - \frac{1}{N}\sum_{i'=1}^N C_{Nk}(i',j) - \frac{1}{N}\sum_{j'=1}^N C_{Nk}(i,j') + \frac{1}{N^2}\sum_{i'=1}^N \sum_{j'=1}^NC_{Nk}(i',j')\\
& =
\begin{cases}
\left[ n_1^{-1}\bm{A}_1 \widecheck{\bm{Y}}_i(1)  - {N^{-1}} \widecheck{\bm{\tau}}_i(\bm{A}) \right]_{(k)},
& 1\leq j \leq n_1, \\
\left[ n_2^{-1}\bm{A}_2 \widecheck{\bm{Y}}_i(2)  - {N^{-1}} \widecheck{\bm{\tau}}_i(\bm{A}) \right]_{(k)}, & 
n_1+1 \leq j \leq n_1+n_2, \\
\quad \quad \quad \vdots & \\
\left[ n_Q^{-1}\bm{A}_Q \widecheck{\bm{Y}}_i(Q) - {N^{-1}} \widecheck{\bm{\tau}}_i(\bm{A})  \right]_{(k)}, & 
\sum_{r=1}^{Q-1}n_r+1 \leq j \leq N.
\end{cases}
\end{align*}
We verify the Condition (b) of Lemma \ref{lemma:fraser}. 
First, we give an upper bound of $\max_{1\leq i,j\leq N}D^2_{Nk}(i,j)$. 
The definition of $D^2_{Nk}(i,j)$ implies 
\begin{align}\label{eq:max_D_rand_exp_proof}
\max_{1\leq i,j\leq N}D^2_{Nk}(i,j) = \max_{1\leq q\leq Q}\max_{1\leq i\leq N} \left[ n_q^{-1} \bm{A}_q \widecheck{\bm{Y}}_i(q)  - N^{-1} \widecheck{\bm{\tau}}_i(\bm{A}) \right]_{(k)}^2.
\end{align}
The key term $[ n_q^{-1} \bm{A}_q \widecheck{\bm{Y}}_i(q)  - N^{-1} \widecheck{\bm{\tau}}_i(\bm{A}) ]_{(k)}^2$ has the following equivalent forms: 
\begin{eqnarray*}
&&\left[ n_q^{-1} \bm{A}_q \widecheck{\bm{Y}}_i(q)  - N^{-1} \widecheck{\bm{\tau}}_i(\bm{A}) \right]_{(k)}^2
 = 
\left[ \frac{1}{n_q}\bm{A}_q \widecheck{\bm{Y}}_i(q)  -  \sum_{r=1}^Q \frac{1}{N} \bm{A}_r \widecheck{\bm{Y}}_i(r) \right]_{(k)}^2 \nonumber\\
& =& \left[ \left( \frac{1}{n_q}-\frac{1}{N}\right)\bm{A}_q \widecheck{\bm{Y}}_i(q) - \sum_{r\neq q}\frac{1}{N}\bm{A}_r \widecheck{\bm{Y}}_i(r)\right]_{(k)}^2
 = \left\{
\left( \frac{1}{n_q}- \frac{1}{N}\right) \left[\bm{A}_q  \widecheck{\bm{Y}}_i(q) \right]_{(k)}-
\sum_{r\neq q}\frac{1}{N}\left[\bm{A}_r \widecheck{\bm{Y}}_i(r)\right]_{(k)} 
\right\}^2,
\end{eqnarray*}
which can be bounded from above by  
\begin{align}\label{eq:cauchy_indivi}
\left[ n_q^{-1} \bm{A}_q \widecheck{\bm{Y}}_i(q)  - N^{-1} \widecheck{\bm{\tau}}_i(\bm{A}) \right]_{(k)}^2 \leq  Q
\left\{
\left( \frac{1}{n_q}- \frac{1}{N}\right)^2 \left[\bm{A}_q  \widecheck{\bm{Y}}_i(q) \right]_{(k)}^2+
\sum_{r\neq q}\frac{1}{N^2}\left[\bm{A}_r \widecheck{\bm{Y}}_i(r)\right]_{(k)}^2 
\right\},
\end{align}
following from the Cauchy--Schwarz inequality.
Because all terms in the curly brackets in formula \eqref{eq:cauchy_indivi} is less than or equal to 
$\max_{1\leq r\leq Q}\{ [\bm{A}_r \widecheck{\bm{Y}}_i(r)]_{(k)}^2/n_r^2\}$, 
the key term $[ n_q^{-1} \bm{A}_q \widecheck{\bm{Y}}_i(q)  - N^{-1} \widecheck{\bm{\tau}}_i(\bm{A}) ]_{(k)}^2$ is bounded from above by 
$$
[ n_q^{-1} \bm{A}_q \widecheck{\bm{Y}}_i(q)  - N^{-1} \widecheck{\bm{\tau}}_i(\bm{A}) ]_{(k)}^2\leq Q^2 \max_{1\leq r\leq Q}\{ [\bm{A}_r \widecheck{\bm{Y}}_i(r)]_{(k)}^2/n_r^2\}.
$$ 
Therefore, according to \eqref{eq:max_D_rand_exp_proof}, 
\begin{align*}
\max_{1\leq i,j\leq N}D^2_{Nk}(i,j) & = \max_{1\leq q\leq Q}\max_{1\leq i\leq N}\left[ n_q^{-1} \bm{A}_q \widecheck{\bm{Y}}_i(q)  - N^{-1} \widecheck{\bm{\tau}}_i(\bm{A}) \right]_{(k)}^2
\leq \max_{1\leq q\leq Q}\max_{1\leq i\leq N}
Q^2\max_{1\leq r\leq Q}\frac{1}{n_r^2} \left[\bm{A}_r \widecheck{\bm{Y}}_i(r)\right]_{(k)}^2
 \\
& = Q^2 \max_{1\leq i\leq N} \max_{1\leq r\leq Q}\frac{1}{n_r^2}\left[\bm{A}_r \widecheck{\bm{Y}}_i(r) \right]_{(k)}^2
=  Q^2\max_{1\leq r\leq Q} \frac{1}{n_r^2} \max_{1\leq i\leq N}\left[\bm{A}_r \widecheck{\bm{Y}}_i(r) \right]_{(k)}^2 \\
& = Q^2\max_{1\leq r\leq Q} \frac{1}{n_r^2}m_r(k), 
\end{align*}
where the last equality follows from the definition of $m_r(k).$ 

Second, we simplify $\sum_{i=1}^N \sum_{j=1}^N  D_{Nk}^2(i,j)$. Following from simple algebra, we have
\begin{align*}
\sum_{i=1}^N \sum_{j=1}^N  D_{Nk}^2(i,j) & = \sum_{q=1}^Q n_q \sum_{i=1}^N \left[ n_q^{-1}\bm{A}_q \widecheck{\bm{Y}}_i(q)  - {N^{-1}} \widecheck{\bm{\tau}}_i(\bm{A}) \right]_{(k)}^2\\
& = \sum_{q=1}^Q n_q \sum_{i=1}^N \left( \left[ n_q^{-1}\bm{A}_q \widecheck{\bm{Y}}_i(q)  \right]_{(k)}^2+\left[ {N^{-1}} \widecheck{\bm{\tau}}_i(\bm{A}) \right]_{(k)}^2 - 2\left[ n_q^{-1}\bm{A}_q \widecheck{\bm{Y}}_i(q)  \right]_{(k)}\left[ {N^{-1}} \widecheck{\bm{\tau}}_i(\bm{A}) \right]_{(k)}\right)\\
& = \sum_{q=1}^Q  \frac{1}{n_q} \sum_{i=1}^N  \left[\bm{A}_q \widecheck{\bm{Y}}_i(q)  \right]_{(k)}^2 + \frac{1}{N} \sum_{i=1}^N \left[  \widecheck{\bm{\tau}}_i(\bm{A}) \right]_{(k)}^2 - 2 \frac{1}{N}\sum_{i=1}^N \sum_{q=1}^Q  \left[\bm{A}_q \widecheck{\bm{Y}}_i(q)  \right]_{(k)}\left[  \widecheck{\bm{\tau}}_i(\bm{A}) \right]_{(k)}\\
& = \sum_{q=1}^Q \frac{1}{n_q} \sum_{i=1}^N  \left[\bm{A}_q \widecheck{\bm{Y}}_i(q)  \right]_{(k)}^2 - \frac{1}{N} \sum_{i=1}^N \left[  \widecheck{\bm{\tau}}_i(\bm{A}) \right]_{(k)}^2. 
\end{align*}
According to the definition of $v_q(k)$ and $v_{\bm{\tau}}(k)$, it further reduces to 
\begin{align*}
\sum_{i=1}^N \sum_{j=1}^N  D_{Nk}^2(i,j)
& = (N-1)
\left( \sum_{q=1}^Q n_q^{-1}v_q(k) - N^{-1}v_{\bm{\tau}}(k) \right).
\end{align*}
Therefore, the quantity in Condition (b) of Lemma \ref{lemma:fraser} satisfies
\begin{align*}
\frac{\max_{i,j}D^2_{Nk}(j,i) }{
N^{-1}\sum_{i=1}^N \sum_{j=1}^N  D_{Nk}^2(i,j)
} & 
\leq  Q^2\cdot\frac{N}{N-1}\max_{1\leq q\leq Q}
\left( \frac{1}{n_q^2}\cdot
\frac{m_q(k)}{
\sum_{r=1}^Q n_r^{-1}v_r(k) - N^{-1}v_{\bm{\tau}}(k)
}\right).
\end{align*}
According to Lemma \ref{lemma:fraser}, Theorem \ref{thm:fclt_cre} holds.
\end{proof}

\begin{proof}[{\bf Proof of Corollary \ref{cor:add_causal_effects}}]
Under the additive causal effects assumption, for any $1\leq q\leq Q$ and $1\leq k\leq K$, we have $v_{\bm{\tau}}(k)=0$. If $v_{q}(k) > 0$, then
\begin{align}\label{eq:inequlaity_proof_exp_add}
\frac{1}{n_q^2}\cdot
\frac{m_q(k)}{
\sum_{r=1}^Q n_r^{-1}v_r(k) - N^{-1}v_{\bm{\tau}}(k)}
= \frac{1}{n_q^2}\cdot
\frac{m_q(k)}{
\sum_{r=1}^Q n_r^{-1}v_r(k)} \leq  \frac{1}{n_q^2}\cdot
\frac{m_q(k)}{
n_q^{-1}v_q(k)} = 
\frac{1}{n_q}
\frac{m_q(k)}{
v_q(k)} .
\end{align}
If $v_q(k)=0$, then $m_q(k)$ must also be zero, and the inequality \eqref{eq:inequlaity_proof_exp_add} still holds provided that we define $m_q(k)/v_q(k)$ as zero in this case.
According to Theorem \ref{thm:fclt_cre}, Corollary \ref{cor:add_causal_effects} holds.

\end{proof}

\begin{proof}[{\bf Proof of Theorem \ref{cor:fclt_exp_stable}}]
Because for any $1\leq q,r\leq Q$, $\bm{S}_q^2$ and $\bm{S}_{qr}$ have limits, $
\bm{S}^2_{\bm{\tau}(\bm{A})}
$
also has a limit according to formula \eqref{eq:S_tau}. 
According to the covariance formula of $\widehat{\bm{\tau}}(\bm{A})$ in Theorem \ref{thm::repeated-sampling}, 
\begin{align*}
N\Cov\left\{
\widehat{\bm{\tau}}(\bm{A})
\right\}  = \sum_{q=1}^Q
\frac{N}{n_q}
\bm{A}_q
\bm{S}^2_q
\bm{A}_q^\top - \bm{S}^2_{\bm{\tau}(\bm{A})}
\end{align*}
has a limiting value $\bm{V}.$
For any $1\leq k,r\leq K$, let $V_{kr}$ be the $(k,r)$ element of $\bm{V}$. Without loss of generality, we assume $V_{kk}\neq 0$ for $1\leq k\leq K_0$, and $V_{kk}=0$ for $K_0<k\leq K$. We partition $\bm{V}$ as
\begin{align*}
\bm{V} = 
\begin{pmatrix}
\bm{V}_1 & \bm{V}_2\\
\bm{V}_2^\top & \bm{V}_3
\end{pmatrix},
\end{align*}
where the diagonal elements of $\bm{V}_1$ are non-zero, and the diagonal elements of $\bm{V}_3$ are zero. Because $\bm{V}$ is the limit of a covariance matrix, for any $1\leq k,r\leq K$, 
the Cauchy--Schwarz inequality implies
$V_{kr}^2 \leq V_{kk}V_{rr}$. Therefore, $\bm{V}_2=\bm{0}_{K_0\times (K-K_0)},$ $\bm{V}_3 = \bm{0}_{(K-K_0)\times (K-K_0)}$, and $\bm{V}=\text{diag}(\bm{V}_1, \bm{0})$.

First, we prove that $\left( \widehat{\tau}_{(1)}(\bm{A}), \ldots, \widehat{\tau}_{(K_0)}(\bm{A}) \right)$ is asymptotically Normal. 
The conditions of Theorem \ref{cor:fclt_exp_stable} implies that for any $1\leq k\leq K_0$, $\max_{1\leq i\leq N} \left\|
\bm{Y}_i(q)-\bar{\bm{Y}}(q) \right\|_2^2/N\rightarrow 0,$ and 
$N\Var\{
\widehat{\tau}_{(k)}(\bm{A}) \}\rightarrow V_{kk}>0.$ 
Moreover, $m_q(k)$ is bounded above by:
\begin{align*}
m_q(k) & = \max_{1\leq i\leq N}\left[
\bm{A}_q\bm{Y}_i(q)-\bm{A}_q\bar{\bm{Y}}(q)
\right]_{(k)}^2 \leq \max_{1\leq i\leq N}\left\| 
\bm{A}_q\bm{Y}_i(q)-\bm{A}_q\bar{\bm{Y}}(q)
\right\|_2^2 \leq 
\left\| \bm{A}_q \right\|_2^2 \max_{1\leq i\leq N} \left\|
\bm{Y}_i(q)-\bar{\bm{Y}}(q) \right\|_2^2,
\end{align*}
where $\| \bm{A} \|_2  = \sup_{\bm{x}\neq \bm{0}}  \| \bm{A} \bm{x} \|_2 / \|  \bm{x} \|_2$ is the $2$-norm of matrix $\bm{A}.$
Therefore, the left-hand side of Condition \eqref{eq:condi_cre} in Theorem \ref{thm:fclt_cre} is bounded by:  
\begin{align*}
\frac{1}{n_q^2}\frac{m_q(k)}{\sum_{r=1}^Qn_r^{-1}v_r(k)-N^{-1}v_{\bm{\tau}}(k)} & =
\frac{m_q(k)}{
n_q^2 \Var\left\{
\widehat{\tau}_{(k)}(\bm{A})
\right\}
}  \leq 
\frac{\left\| \bm{A}_q \right\|_2^2\max_{1\leq i\leq N} \left\|
\bm{Y}_i(q)-\bar{\bm{Y}}(q) \right\|_2^2}{
n_q^2 \Var\left\{
\widehat{\tau}_{(k)}(\bm{A})
\right\}
}\\
& =
  \left( \frac{N}{n_q}\right)^2 \frac{\left\| \bm{A}_q \right\|_2^2}{N\Var\left\{
\widehat{\tau}_{(k)}(\bm{A})
\right\}} \frac{\max_{1\leq i\leq N} \left\|
\bm{Y}_i(q)-\bar{\bm{Y}}(q) \right\|_2^2}{N},
\end{align*}
which converges to zero because of the conditions in Theorem \ref{cor:fclt_exp_stable}. 
Note that
the correlation matrix of $\left( \widehat{\tau}_{(1)}(\bm{A}), \ldots, \widehat{\tau}_{(K_0)}(\bm{A}) \right)$ has limit $\text{diag}(\bm{V}_1)^{-1/2}\bm{V}_1 \text{diag}(\bm{V}_1)^{-1/2}$. 
According to Theorem \ref{thm:fclt_cre}, 
\begin{align*}
\left( \frac{\widehat{\tau}_{(1)}(\bm{A})- {\tau}_{(1)}(\bm{A})}{\sqrt{\Var\left\{ \widehat{\tau}_{(1)}(\bm{A}) \right\}}}, \ldots, \frac{\widehat{\tau}_{(K_0)}(\bm{A}) - {\tau}_{(K_0)}(\bm{A})}{\sqrt{\Var\left\{
\widehat{\tau}_{(K_0)}(\bm{A})
\right\}}}
\right) \converged \mathcal{N}\left( \bm{0}, \text{diag}(\bm{V}_1)^{-1/2}\bm{V}_1 \text{diag}(\bm{V}_1)^{-1/2}\right).
\end{align*}
Using Slutsky's theorem, $\sqrt{N}\left( \widehat{\tau}_{(1)}(\bm{A})-{\tau}_{(1)}(\bm{A}), \ldots, \widehat{\tau}_{(K_0)}(\bm{A}) - {\tau}_{(K_0)}(\bm{A}) \right) \converged \mathcal{N}(\bm{0}, \bm{V}_1)$. 

Second, we prove that $\widehat{\bm{\tau}}(\bm{A})$ is asymptotically Normal.
For any $K_0< k \leq K$, because $N\Var\left\{ \widehat{\tau}_{(k)}(\bm{A}) \right\}\rightarrow V_{kk}=0$,  by the Markov inequality, 
$$
\sqrt{N}\left\{ {\widehat{\tau}_{(k)}(\bm{A})- {\tau}_{(k)}(\bm{A})} \right\}  = O_p\left( {\sqrt{N\Var\left\{ \widehat{\tau}_{(k)}(\bm{A}) \right\}}}\right) = o_p(1).
$$
Therefore, $\sqrt{N}\left( \widehat{\tau}_{(1)}(\bm{A})-{\tau}_{(1)}(\bm{A}), \ldots, \widehat{\tau}_{(K)}(\bm{A}) - {\tau}_{(K)}(\bm{A}) \right)$ is asymptotically Normal with mean $\bm{0}$ and covariance matrix $\text{diag}(\bm{V}_1, \bm{0})=\bm{V}$.
\end{proof}

\begin{proof}[{\bf Comments on Theorem \ref{cor:fclt_exp_stable}}]
When each coordinates of the $\bm{Y}_i(q)$'s are bounded, the
$\left\|\bm{Y}_i(q) - \bar{\bm{Y}}(q)\right\|_2^2$'s are also bounded for all units, and therefore
the regularity condition  $\max_{1\leq i \leq N}\left\|\bm{Y}_i(q) - \bar{\bm{Y}}(q)\right\|_2^2/N \rightarrow 0$ must hold. 
When the coordinates of the $\bm{Y}_i(q)$'s are i.i.d draws from a superpopulation with more than two moments, according to the comments on Theorem \ref{thm:CLThajek}, for each coordinate $1\leq k\leq p$, $\max_{1\leq i\leq N}[\bm{Y}_i(q)-\bar{\bm{Y}}(q)]_{(k)}^2/N\rightarrow 0$ almost surely. Moreover, because 
\begin{align*}
\frac{1}{N}\max_{1\leq i \leq N}\left\|\bm{Y}_i(q) - \bar{\bm{Y}}(q)\right\|_2^2 = 
\frac{1}{N}\max_{1\leq i \leq N}\left\{
\sum_{k=1}^p [\bm{Y}_i(q)-\bar{\bm{Y}}(q)]_{(k)}^2
\right\} \leq
\sum_{k=1}^p \frac{1}{N} \max_{1\leq i \leq N}
 [\bm{Y}_i(q)-\bar{\bm{Y}}(q)]_{(k)}^2,
\end{align*}
the regularity condition $\max_{1\leq i \leq N}\left\|\bm{Y}_i(q) - \bar{\bm{Y}}(q)\right\|_2^2/N \rightarrow 0$ holds almost surely. 
\end{proof}

To prove Proposition \ref{prop:con_set_tau}, we need the following lemma. For any $1\leq q \leq Q$ and $1\leq k,k'\leq p$, let $B_i$ and $C_i$ be the $k$th and $k'$th coordinates of $\bm{Y}_i(q)-\bar{\bm{Y}}(q)$, respectively, and $S_{B,C}=\sum_{i=1}^N B_iC_i/(N-1)$ be the finite population covariance between $B_i$ and $C_i$.
Let $\bar{B}_{(q)} = \sum_{i:L_{i}=q}B_{i}/n_q$ and $\bar{C}_{(q)} = \sum_{i:L_{i}=q}C_{i}/n_q$ be the averages of $B_i$ and $C_i$ in treatment group $q$, and
$$
s_{B,C} = \frac{1}{n_q-1 }\sum_{i:L_i=q}
\left(B_i
- \bar{B}_{(q)}
\right)
\left(
C_i- \bar{C}_{(q)}
\right)
$$
be the sample covariance between $B_i$ and $C_i$ in treatment group $q$.
\begin{lemma}\label{lemma:var_est_exp}
If the regularity conditions in  Theorem \ref{cor:fclt_exp_stable} hold, then
$ s_{B,C}- S_{B,C}=o_p(1).
$
\end{lemma}
\begin{proof}[Proof of Lemma  \ref{lemma:var_est_exp}]
Let $S^2_{B}$ and $S^2_{C}$ be the finite population variances of $B_i$ and $C_i$, 
and $S^2_{B\times C}$ be the finite population variance of $B_i\times C_i$; 
$S^2_{B}, S^2_{C}, S_{B,C}$ and $S^2_{B\times C}$ all depends on $N$ implicitly.
Because the finite population means of $B_i$ and $C_i$ are both zeros,
the Markov inequality implies
\begin{align*}
\bar{B}_{(q)} = O_p\left\{ \sqrt{\Var\left( \bar{B}_{(q)} \right)} \right\} = O_p\left( \sqrt{ {n_q^{-1}}S_{B}^2 }\right). 
\end{align*}
Similarly, we also have: 
\begin{align*}
\bar{C}_{(q)}  =
O_p\left( \sqrt{ {n_q^{-1}}S_{C}^2 }\right), \quad 
\frac{1}{n_q}\sum_{i:{L}_{i}=q}B_{i}C_i - \frac{N-1}{N}S_{B,C} = O_p\left(
\sqrt{ n_q^{-1}S^2_{B\times C}}
\right).
\end{align*}
Replacing these terms in the definition of the sample covariance $s_{B,C}$, we have
\begin{align}\label{eq:diff_sBC}
s_{B,C}- S_{B,C}
& = \frac{n_q}{n_q-1}\frac{1}{n_q}\sum_{i:{L}_i=q}
B_i C_i - \frac{n_q}{n_q-1}
\bar{B}_{(q)}
\bar{C}_{(q)} - S_{B,C} \nonumber\\
& =  \frac{n_q}{n_q-1}\frac{N-1}{N}S_{B,C} + 
O_p\left(
\sqrt{ n_q^{-1}S^2_{B\times C} }
\right)+
O_p\left( n_q^{-1}\sqrt{ S_{B}^2S_{C}^2 }\right) - S_{B,C} \nonumber\\
& = O_p\left({n_q^{-1}}S_{B,C}\right)+O_p\left(
\sqrt{ n_q^{-1}S^2_{B\times C} }
\right)+
O_p\left( n_q^{-1}\sqrt{ S_{B}^2S_{C}^2 }\right).
\end{align}
According to the regularity conditions in Theorem \ref{cor:fclt_exp_stable}, $n_q/N$ has a positive limiting value, and $S_C^2, S_B^2$ and $S_{B,C}$ have limiting values, 
both ${n_q^{-1}}S_{B,C}$ and $n_q^{-1}\sqrt{ S_{B}^2S_{C}^2 }$ is of order $O_p(N^{-1}).$ 
The key of the second term $S^2_{B\times C}$ in \eqref{eq:diff_sBC} has the following upper bound:
\begin{align*}
S^2_{B\times C}\leq \frac{1}{N-1}\sum_{i=1}^N B_i^2C_i^2 \leq \left( \max_{1\leq i\leq N}B_i^2 \right)\left( \frac{1}{N-1}\sum_{i=1}^N C_i^2 \right) = \max_{1\leq i\leq N}B_i^2 \cdot S_C^2.
\end{align*}
According to the regularity conditions in Theorem \ref{cor:fclt_exp_stable}, $\max_{1\leq i\leq N}B_i^2/N$ converges to $0$, and therefore,  $O_p\left(
\sqrt{ n_q^{-1}S^2_{B\times C} }
\right)=o_p(1)$.
Thus,
$
s_{B,C}- S_{B,C}
 = o_p(1).
$
\end{proof}

\begin{proof}[{\bf Proof of Proposition \ref{prop:con_set_tau}}]
According to Lemma \ref{lemma:var_est_exp}, for any $1\leq q\leq Q$, each element of $\bm{s}_q^2 - \bm{S}_q^2$ converges in probability to zero. Let $\bm{V}_0$ be the limit of $N\sum_{q=1}^Q n_q^{-1}\bm{A}_q\bm{S}_q\bm{A}_q$. Then $\bm{V}_0 - \bm{V}=\lim_{N\rightarrow\infty}\bm{S}^2_{\bm{\tau}(\bm{A})}\geq \bm{0}$. Using the continuous mapping theorem, $N\widehat{\bm{V}}_{A}-\bm{V}_0\overset{p}{\longrightarrow} \bm{0}$. Let $\bm{D}$ be a $K$ dimensional standard Normal random vector, and $\bm{D}^\top \bm{D}$ be a $\chi_K^2$ random variable. Because $\bm{V}_0\geq \bm{V}$, we have 
$\bm{V}_0^{-1}\leq \bm{V}^{-1}$, which further implies $ \bm{V}^{1/2}  \bm{V}_0^{-1}  \bm{V}^{1/2} \leq \bm{I}$ and $\bm{D}^\top \bm{V}^{1/2}\bm{V}_0^{-1}\bm{V}^{1/2}\bm{D} \leq \bm{D}^\top \bm{D}.$
Slutsky's theorem implies
\begin{align*}
\left\{ \widehat{\bm{\tau}}(\bm{A}) - {\bm{\tau}}(\bm{A}) \right\}^\top 
\widehat{V}_{\bm{A}}^{-1}
\left\{ \widehat{\bm{\tau}}(\bm{A}) - {\bm{\tau}}(\bm{A}) \right\} 
& = 
\sqrt{N}\left\{ \widehat{\bm{\tau}}(\bm{A}) - {\bm{\tau}}(\bm{A}) \right\}^\top 
\left( N\widehat{V}_{\bm{A}} \right)^{-1}
\sqrt{N}\left\{ \widehat{\bm{\tau}}(\bm{A}) - {\bm{\tau}}(\bm{A}) \right\} \\
& \converged  \bm{D}^\top \bm{V}^{1/2}\bm{V}_0^{-1}\bm{V}^{1/2}\bm{D}.
\end{align*}
Therefore, the coverage rate of the confidence set satisfies
\begin{eqnarray}\label{eq:coverge_limit}
&& \lim_{N\rightarrow\infty}P\left[
\left\{ \widehat{\bm{\tau}}(\bm{A}) - {\bm{\tau}}(\bm{A}) \right\}^\top 
\widehat{V}_{\bm{A}}^{-1}
\left\{ \widehat{\bm{\tau}}(\bm{A}) - {\bm{\tau}}(\bm{A}) \right\} \leq q_{K, 1-\alpha}
\right] 
=
P\left\{
\bm{D}^\top \bm{V}^{1/2}\bm{V}_0^{-1}\bm{V}^{1/2}\bm{D} \leq q_{K, 1-\alpha}
\right\}\nonumber\\
& \geq & P\left\{
\bm{D}^\top \bm{D} \leq q_{K, 1-\alpha}
\right\} = 1-\alpha.
\end{eqnarray}
When the causal effects are additive, $\bm{V}_0 - \bm{V}= \bm{0}$, and therefore the equality in \eqref{eq:coverge_limit} holds.
\end{proof}

\section{More on the examples}\label{app:more_eg}
\begin{proof}[{\bf Example \ref{eg:hyper}}]
The finite population $\Pi_N=\{y_{Ni}: 1\leq i \leq N\}$ with $N_1$ being $1$ and $N-N_1$ being $0$ has variance
$
v_N = {N_1(N-N_1)}/\{N(N-1)\},
$
and the maximum squared distance of the $y_{Ni}$'s from the population mean $\bar{y}_N = N_1/N$ is
$
m_N = \max\{N_1^2, (N-N_1)^2\}/N^2.
$
The Hypergeometric random variable, $n \ybars $, has variance $\Var(n \ybars) = N_1(N-N_1)n(N-n)/\{N^2(N-1)\}$. Therefore, if $\Var(n \ybars)\rightarrow \infty$, then Condition \eqref{eq:condi_CLThajek} holds:
\begin{eqnarray*}
&& \frac{1}{\min(n, N-n)}\cdot\frac{  m_N }{ v_N }  = \frac{1}{\min(n, N-n)}
\frac{N-1}{N}\max\left(
\frac{N_1}{N-N_1}, \frac{N-N_1}{N_1}
\right)\\
& \leq& 
\left(
\frac{1}{n} + \frac{1}{N-n}
\right)\cdot
\frac{N-1}{N}\cdot
\left(
\frac{N}{N-N_1} + \frac{N}{N_1}
\right)
=\frac{N-1}{n(N-n)}\frac{N^2}{N_1(N-N_1)} = \Var(n \ybars)^{-1}
\rightarrow 0. 
\end{eqnarray*} 
\end{proof}

\begin{proof}[{\bf Example \ref{eg:instru_estimation}}]
First, we show that Condition \eqref{eq:condi_CLThajek} for $\{A_i:i=1,\ldots,N\}$ holds if $\{Y_i:i=1,\ldots,N\}$ and $\{D_i:i=1,\ldots,N\}$ satisfy condition \eqref{eq:condi_instru_Y_D} introduced in the main text.
For any $\beta$, the Cauchy--Schwarz inequality implies that 
\begin{align*}
\max_{1\leq i\leq N}(A_i-\bar{A})^2 & = \max_{1\leq i\leq N} \{ Y_i - \bar{Y}-\beta (D_i-\bar{D}) \}^2 \leq 2\max_{1\leq i\leq N} (Y_i-\bar{Y})^2 + 2\beta^2 \max_{1\leq i\leq N} (D_i-\bar{D})^2.
\end{align*}
The sample variance of the $A_i$'s can be bounded by 
\begin{align*}
s_A^2 & = s_Y^2 + \beta^2s_D^2 - 2 \beta s_{YD} = s_Y^2 - s_{YD}^2/s_D^2 + \left(\beta s_D - s_{YD}/s_D \right)^2 \geq s_Y^2 - s_{YD}^2/s_D^2,  \\
s_A^2 & = s_Y^2 + \beta^2s_D^2 - 2 \beta s_{YD} = 
\beta^2(s_D^2-s_{YD}^2/s_Y^2) + (s_Y -\beta s_{YD}/s_Y)^2\geq 
\beta^2(s_D^2-s_{YD}^2/s_Y^2).
\end{align*}
The above three inequalities imply
\begin{align*}
\frac{\max_{i}(A_i-\bar{A})^2}{s_A^2} & \leq \frac{2\max_i (Y_i-\bar{Y})^2}{s_A^2} + \frac{2\beta^2\max_i (D_i-\bar{D})^2}{s_A^2} \leq \frac{2\max_i (Y_i-\bar{Y})^2}{s_Y^2 - s_{YD}^2/s_D^2} + \frac{2\max_i (D_i-\bar{D})^2}{s_D^2-s_{YD}^2/s_Y^2}.
\end{align*}
Therefore, Condition \eqref{eq:condi_CLThajek} for $\{A_i:1,\ldots,N\}$ holds if \eqref{eq:condi_instru_Y_D} holds.

Second, we study the form of the confidence set.
Recall that $\eta = {N}/{(n_1 n_0)} \cdot\{  \Phi^{-1}(\alpha/2) \}^2$. The inequality \eqref{eq::ci-fieller-creasy} is equivalent to 
\begin{align}\label{eq:beta_quad}
(\widehat{\tau}_D^2 - \eta s_D^2)\beta^2 - 2(\widehat{\tau}_D\widehat{\tau}_Y - \eta s_{YD})\beta + (\widehat{\tau}_Y^2 - \eta s_Y^2) \leq 0. 
\end{align}
The form of the confidence set for $\beta$, or equivalently the solution of quadratic inequality \eqref{eq:beta_quad}, depends on the signs of $\widehat{\tau}_D^2 - \eta s_D^2, \widehat{\tau}_D\widehat{\tau}_Y - \eta s_{YD}, \widehat{\tau}_Y^2 - \eta s_Y^2$ and $\Delta = 4(\widehat{\tau}_D\widehat{\tau}_Y -  \eta s_{YD})^2 - 4(\widehat{\tau}_D^2 - \eta s_D^2)(\widehat{\tau}_Y^2 - \eta s_Y^2)$. 
In addition to the scenarios discussed in the main text, 
Table \ref{tab:form_ci_beta} summarizes all the scenarios.
\begin{table}[htbp]
\centering
\caption{Forms of the confidence sets for $\beta$, where $c_1<c_2$ denote two roots of \eqref{eq:beta_quad} if they exist, and $c$ denotes the only root of \eqref{eq:beta_quad}  if it exists.
}
\label{tab:form_ci_beta}
\begin{tabular}{cccccc}
\toprule
$\widehat{\tau}_D^2 - \eta s_D^2$ & $\Delta$ & $\widehat{\tau}_D\widehat{\tau}_Y - \eta s_{YD}$ & $\widehat{\tau}_Y^2 - \eta s_Y^2$ & form of the confidence set for $\beta$\\
\midrule
$> 0$ & $<0$ & & & empty set \\
$> 0$ & $=0$ & & & one point \\
$> 0$ & $>0$ & & & $[c_1, c_2]$ \\
$<0$ & $\leq 0$ & & & the whole real line\\
$<0$ & $>0$ & & & $(-\infty, c_1]\bigcup [c_2, \infty)$
\\
$=0$ & & $>0$ & & $[c, \infty)$\\
$=0$ & & $<0$ & & $(-\infty, c]$\\
$=0$ & & $=0$ & $>0$&  empty set \\
$=0$ & & $=0$ & $\leq 0$&  the whole real line \\
\bottomrule
\end{tabular}
\end{table}
\end{proof}

\begin{proof}[{\bf Example \ref{eg:regression_adjust}}]


We provide the technical details of Example \ref{eg:regression_adjust} highlighting the following five points.
%
%
%
%
%

\paragraph{(9.1)}
The variance of $\widehat{\tau}({\widetilde{\bm{\beta}}_1, \widetilde{\bm{\beta}}_0})$ is smaller than $\widehat{\tau}({{\bm{\beta}}_1, {\bm{\beta}}_0})$ for any constant vectors ${\bm{\beta}}_1$ and ${\bm{\beta}}_0$. Specifically, 
$\Var\{
\widehat{\tau}({{\bm{\beta}}_1, {\bm{\beta}}_0})
\} =  \Var\{
\widehat{\tau}({\widetilde{\bm{\beta}}_1, \widetilde{\bm{\beta}}_0})
\}  +\Var\{
\widehat{\tau}({{\bm{\beta}}_1, {\bm{\beta}}_0}) - \widehat{\tau}({\widetilde{\bm{\beta}}_1, \widetilde{\bm{\beta}}_0})
\}.$

For notational simplicity, we use $\Var_{\text{f}}(\cdot)$ to denote the finite population variance of quantities of the finite units. 
For example, for $N$ units with values $y_{Ni}$'s, $\Var_{\text{f}}(y_{Ni})=\sum_{i=1}^N(y_{Ni}-\bar{y}_N)^2/(N-1)$, where $\bar{y}_N$ is the average of the $y_{Ni}$'s.
According to \citet{Neyman:1923} or
Theorem \ref{thm::repeated-sampling}, 
\begin{align*}
\Var\left\{
\widehat{\tau}({{\bm{\beta}}_1, {\bm{\beta}}_0})
\right\} & = \frac{1}{n_1}\Var_{\text{f}}\left\{ {Y}_i(1)-\bm{\beta}_1^\top {\bm{X}}_i \right\}+ \frac{1}{n_0}\Var_{\text{f}}\left\{
{Y}_i(0)
-\bm{\beta}_0^\top {\bm{X}}_i \right\} -\frac{1}{N}\Var_{\text{f}}\left\{ {\tau}_i -(\bm{\beta}_1-\bm{\beta_0})^\top{\bm{X}}_i\right\}.
\end{align*}
By definition, $\widetilde{\bm{\beta}}_1$ is the coefficient of the linear projection of $Y(1)$ onto the space of $\bm{X}$, and therefore
\begin{align*}
\Var_{\text{f}}\left\{ 
{Y}_i(1)-\bm{\beta}_1^\top{\bm{X}}_i\right\}  = 
\Var_{\text{f}}\left\{ {Y}_i(1)-\widetilde{\bm{\beta}}_1^\top {\bm{X}}_i + (\widetilde{\bm{\beta}}_1 - \bm{\beta}_1)^\top {\bm{X}}_i \right\} 
 = \Var_{\text{f}}\left\{ {Y}_i(1)-\widetilde{\bm{\beta}}_1^\top {\bm{X}}_i\right\} + \Var_{\text{f}}\left\{  (\widetilde{\bm{\beta}}_1 - \bm{\beta}_1)^\top {\bm{X}}_i \right\}.
\end{align*}
Similarly,
\begin{align*}
\Var_{\text{f}}\left\{ {Y}_i(0)-\bm{\beta}_0^\top {\bm{X}}_i \right\} 
& = \Var_{\text{f}}\left\{ {Y}_i(0)-\widetilde{\bm{\beta}}_0^\top {\bm{X}}_i  \right\} + \Var_{\text{f}}\left\{  (\widetilde{\bm{\beta}}_0 - \bm{\beta}_0)^\top {\bm{X}}_i \right\},\\
\Var_{\text{f}}\left\{ {\tau}_i -(\bm{\beta}_1-\bm{\beta_0})^\top {\bm{X}}_i \right\} & = 
\Var_{\text{f}}\left\{ {\tau}_i -(\widetilde{\bm{\beta}}_1-\widetilde{\bm{\beta_0}})^\top {\bm{X}}_i \right\} + 
\Var_{\text{f}}\left[
\left\{
(\widetilde{\bm{\beta}}_1 - \bm{\beta}_1) - 
(\widetilde{\bm{\beta}}_0 - \bm{\beta}_0)
\right\}^\top {\bm{X}}_i
\right].
\end{align*}
Therefore, 
\begin{align}\label{eq:var_decomp_proof}
& \quad \Var\left\{
\widehat{\tau}({{\bm{\beta}}_1, {\bm{\beta}}_0})
\right\}  - \Var\left\{
\widehat{\tau}({\widetilde{\bm{\beta}}_1, \widetilde{\bm{\beta}}_0})
\right\} \nonumber \\
& = \frac{1}{n_1}\Var_{\text{f}}\left\{  (\widetilde{\bm{\beta}}_1 - \bm{\beta}_1)^\top {\bm{X}}_i \right\} + \frac{1}{n_0}\Var_{\text{f}}\left\{  (\widetilde{\bm{\beta}}_0 - \bm{\beta}_0)^\top {\bm{X}}_i \right\} -\frac{1}{N}\Var_{\text{f}}\left[
\left\{
(\widetilde{\bm{\beta}}_1 - \bm{\beta}_1) - 
(\widetilde{\bm{\beta}}_0 - \bm{\beta}_0)
\right\}^\top {\bm{X}}_i
\right]\nonumber
\\
& = \Var\left\{
\frac{1}{n_1}\sum_{i=1}^N Z_i(\widetilde{\bm{\beta}}_1 - \bm{\beta}_1)^\top {\bm{X}}_i
- \frac{1}{n_0}\sum_{i=1}^N (1-Z_i)(\widetilde{\bm{\beta}}_0 - \bm{\beta}_0)^\top {\bm{X}}_i
\right\} \\
& = \Var\left\{
\widehat{\tau}({{\bm{\beta}}_1, {\bm{\beta}}_0}) - \widehat{\tau}({\widetilde{\bm{\beta}}_1, \widetilde{\bm{\beta}}_0})
\right\}\geq 0, \nonumber
\end{align}
where \eqref{eq:var_decomp_proof} follows from \citet{Neyman:1923} or
the formula in Theorem \ref{thm::repeated-sampling}, considering the finite population with treatment potential outcome $(\widetilde{\bm{\beta}}_1 - \bm{\beta}_1)^\top {\bm{X}}_i$ and control potential outcome $(\widetilde{\bm{\beta}}_0 - \bm{\beta}_0)^\top {\bm{X}}_i$ for each unit $i$.

\paragraph{(9.2)} 
$\widehat{\tau}({{\bm{\beta}}_1, {\bm{\beta}}_0})$ is asymptotically Normal, for any vectors ${\bm{\beta}}_1$ and ${\bm{\beta}}_0$ that do not depend on the treatment indicators but have limiting values. In particular, because the finite population covariances among potential outcomes and covariates have limiting values, 
$\widetilde{\bm{\beta}}_1$ and $\widetilde{\bm{\beta}}_0$ have limiting values, and therefore $\widehat{\tau}({\widetilde{\bm{\beta}}_1, \widetilde{\bm{\beta}}_0})$ is asymptotically Normal.

Let $e_i(z)={Y}_i(z) - {\bm{\beta}}_z^\top {\bm{X}}_i$ be the ``adjusted'' potential outcome under treatment $z$, and $\bar{e}(z)$ be the finite population average of the $e_i(z)$'s.
Define $ \bm{\beta}_z = (\beta_{z1}, \ldots, \beta_{zK}) $.
According to the regularity conditions in Example \ref{eg:regression_adjust}, the finite population variances and covariance of $\{e_i(1)\}_{i=1}^N$ and $\{e_i(0)\}_{i=1}^N$ have limiting values, and
\begin{align*}
\frac{1}{N}\max_{1\leq i \leq N}
\left\{
e_i(z)-\bar{e}(z)
\right\}^2 
& =  
\frac{1}{N}\max_{1\leq i \leq N}
\left\{
{Y}_i(z) - \bar{Y}(z) - \sum_{k=1}^K {\beta}_{zk} {X}_{ki} 
\right\}^2 \\
& \leq \frac{1}{N}\max_{1\leq i \leq N}(K+1)
\left[
\left\{{Y}_i(z) - \bar{Y}(z)\right\}^2 + \sum_{k=1}^K {\beta}_{zk}^2 {X}_{ki}^2
\right]\\
& \leq 
\frac{K+1}{N}\max_{1\leq i \leq N}\left\{{Y}_i(z) - \bar{Y}(z)\right\}^2+ \frac{K+1}{N}
\sum_{k=1}^K {\beta}_{zk}^2 \cdot \max_{1\leq i \leq N}{X}_{ki}^2 \rightarrow 0.
\end{align*}
Therefore, the asymptotic Normality of $\widehat{\tau}({{\bm{\beta}}_1, {\bm{\beta}}_0}) $ follows from Theorem \ref{cor:fclt_exp_stable} and 
$$
\sqrt{N}\{ \widehat{\tau}({{\bm{\beta}}_1, {\bm{\beta}}_0}) - \tau \} =
\sqrt{N}\left[ \frac{1}{n_1}\sum_{i=1}^N Z_i e_i(1) - 
\frac{1}{n_0}\sum_{i=1}^N (1-Z_i)e_i(0) - \left\{ \bar{e}(1)-\bar{e}(0)\right\} \right].
$$

\paragraph{(9.3)}
$\widehat{\tau}({\widehat{\bm{\beta}}_1, \widehat{\bm{\beta}}_0})$ and $\widehat{\tau}({\widetilde{\bm{\beta}}_1, \widetilde{\bm{\beta}}_0})$ have the same asymptotic distribution.

It suffices to show  the difference between $\widehat{\tau}({\widehat{\bm{\beta}}_1, \widehat{\bm{\beta}}_0})$ and $\widehat{\tau}({\widetilde{\bm{\beta}}_1, \widetilde{\bm{\beta}}_0})$ is $o_p\left(N^{-1/2}\right)$.
Let $\bar{Y}_T$ and $\bar{Y}_C$ be the averages of the observed outcomes, and $\bar{\bm{X}}_T$ and $\bar{\bm{X}}_C$ be the averages of the covariates in treatment and control groups.
The difference between the two estimators is
\begin{align*}
\widehat{\tau}({\widehat{\bm{\beta}}_1, \widehat{\bm{\beta}}_0}) - \widehat{\tau}({\widetilde{\bm{\beta}}_1, \widetilde{\bm{\beta}}_0})
& =  -\left(
\widehat{\bm{\beta}}_1-{\bm{\beta}}_1
\right)^\top
\bar{\bm{X}}_T  +
\left(
\widehat{\bm{\beta}}_0-{\bm{\beta}}_0
\right)^\top
\bar{\bm{X}}_C.
\end{align*}
The Markov inequality implies, for any $1\leq k\leq K$, $\left[ \bar{\bm{X}}_T  \right]_{(k)} = O_p(N^{-1/2}),$
and therefore
$
\bar{\bm{X}}_T  = O_p(N^{-1/2}).
$
Note that 
\begin{align*}
\widehat{\bm{\beta}}_1 & = \left\{ \frac{1}{n_1-1}\sum_{i=1}^N Z_i (\bm{X}_i -\bar{\bm{X}}_T) (\bm{X}_i -\bar{\bm{X}}_T)^\top \right\}^{-1}
\left\{ \frac{1}{n_1-1}\sum_{i=1}^N Z_i(\bm{X}_i -\bar{\bm{X}}_T)({Y}_i -\bar{Y}_T ) \right\}.
\end{align*}
According to Lemma \ref{lemma:var_est_exp},
the difference between the sample covariance and the corresponding finite population covariance is of order $o_p(1).$
Thus
$\widehat{\bm{\beta}}_1-{\bm{\beta}}_1=o_p(1)$. Similarly, 
$
\bar{\bm{X}}_C  = O_p(N^{-1/2})
$ and $\widehat{\bm{\beta}}_0-{\bm{\beta}}_0=o_p(1)$. Above all, $\widehat{\tau}({\widehat{\bm{\beta}}_1, \widehat{\bm{\beta}}_0}) - \widehat{\tau}({\widetilde{\bm{\beta}}_1, \widetilde{\bm{\beta}}_0})=o_p(N^{-1/2})$.
According to Slutsky's theorem, 
$\sqrt{N}\{\widehat{\tau}({\widehat{\bm{\beta}}_1, \widehat{\bm{\beta}}_0})-\tau\}$ and $\sqrt{N}\{\widehat{\tau}({\widetilde{\bm{\beta}}_1, \widetilde{\bm{\beta}}_0})-\tau\}$ have the same asymptotic distribution.

\paragraph{(9.4)}
We consider estimating the sampling variance and constructing confidence intervals based on the regression adjustment estimator with vectors $(\bm{\beta}_1, \bm{\beta}_0)$ that do not depend on the treatment indicators.

Let
$S^2_z(\bm{\beta}_z)$ be the finite population variance of ``adjusted'' potential outcomes $Y_i(1) - \bm{\beta}_z^\top \bm{X}_i$'s, and $s^2_z(\bm{\beta}_z)$ be the sample variance of $Y_i - \bm{\beta}_z^\top \bm{X}_i$ in treatment arm $z$. 
According to Proposition \ref{prop:con_set_tau}, 
 $s^2_z(\bm{\beta}_z)-S^2_z(\bm{\beta}_z)=o_p(1),$ and therefore the variance estimator $\widehat{V}(\bm{\beta}_1, \bm{\beta}_0)=s^2_1(\bm{\beta}_1)/n_1+s^2_0(\bm{\beta}_0)/n_0$ is asymptotically conservative. If the limits of $S^2_1(\bm{\beta}_1)$ and $S^2_0(\bm{\beta}_0)$ are not both zero,  
the confidence interval $|\widehat{\tau}(\bm{\beta}_1, \bm{\beta}_0)-\tau|\leq |\Phi^{-1}(\alpha/2)| \widehat{V}^{1/2}(\bm{\beta}_1, \bm{\beta}_0)$ has asymptotic coverage rate at least as large as $1-\alpha$. 
When the difference between ``adjusted'' treatment and control potential outcomes,  $\tau_i-(\bm{\beta}_1-\bm{\beta}_0)^\top \bm{X}_i$,  is constant for all units, both the variance estimator and confidence intervals become asymptotically exact.

\paragraph{(9.5)}
We consider the variance estimation for the regression adjustment estimator with sample least squares coefficients $\widehat{\bm{\beta}}_1$ and $\widehat{\bm{\beta}}_0$, and the corresponding confidence intervals. 
We are going to show that the variance estimator and confidence intervals, constructed by treating $\widehat{\bm{\beta}}_1$ and $\widehat{\bm{\beta}}_0$ as predetermined constant vectors, are still asymptotically conservative.

Let $S_{Y(z)}^2, \bm{S}_{\bm{X}}^2$ and $\bm{S}_{\bm{X}Y(z)}=\bm{S}_{Y(z)\bm{X}}^\top$ be the finite population variance and covariance between potential outcomes and covariates, and
$s_{Y,z}^2, \bm{s}_{\bm{X},z}^2$ and $\bm{s}_{\bm{X}Y,z}=\bm{s}_{Y\bm{X},z}^\top$ be the sample variance and covariance between $Y_i$ and $\bm{X}_i$ in treatment arm $z$. 
The finite population variance of $Y_i(z)-\widetilde{\bm{\beta}}_z^\top \bm{X}_i$ can be represented as $
S_z^2(\widetilde{\bm{\beta}}_z) = S_{Y(z)}^2 -  \bm{S}_{Y(z)\bm{X}}\left( \bm{S}_{\bm{X},z}^2 \right)^{-1}\bm{S}_{\bm{X}Y(z)}$, and 
the sample variance of $Y_i-\widehat{\bm{\beta}}_z^\top \bm{X}_i$ in treatment arm $z$ can be represented as
$
s_z^2(\widehat{\bm{\beta}}_z) = s_{Y,z}^2 -  \bm{s}_{Y\bm{X},z}\left( \bm{s}_{\bm{X},z}^2 \right)^{-1}\bm{s}_{\bm{X}Y,z}
$. 
According to Lemma \ref{lemma:var_est_exp}, $s_z^2(\widehat{\bm{\beta}}_z)-S_z^2(\widetilde{\bm{\beta}}_z)=o_p(1)$. 
Note that $\widehat{\tau}({\widehat{\bm{\beta}}_1, \widehat{\bm{\beta}}_0})$ and $\widehat{\tau}({\widetilde{\bm{\beta}}_1, \widetilde{\bm{\beta}}_0})$ have the same asymptotic distribution. The variance estimator 
$\widehat{V}(\widehat{\bm{\beta}}_1, \widehat{\bm{\beta}}_0)=s^2_1(\widehat{\bm{\beta}}_1)/n_1+s^2_0(\widehat{\bm{\beta}}_0)/n_0$ is asymptotically conservative, and 
it is asymptotically equivalent to the Huber--White variance estimator \citep{lin2013}. Following the same logic as Proposition \ref{prop:con_set_tau}, if $S_1^2(\widetilde{\bm{\beta}}_1)$ and $S_0^2(\widetilde{\bm{\beta}}_0)$ are not both zero, the confidence interval 
$|\widehat{\tau}(\widehat{\bm{\beta}}_1, \widehat{\bm{\beta}}_0)-\tau|\leq |\Phi^{-1}(\alpha/2)| \widehat{V}^{1/2}(\widehat{\bm{\beta}}_1, \widehat{\bm{\beta}}_0)$
has asymptotic coverage rate at least as large as $1-\alpha$. If the residual from the linear projection of $\tau_i$ on $\bm{X}_i$,
$\tau_i-(\widetilde{\bm{\beta}}_1-\widetilde{\bm{\beta}}_0)^\top \bm{X}_i$, is constant for all units, then both the variance estimator and confidence intervals become asymptotically exact. 
\end{proof}

\begin{proof}[{\bf Example \ref{eg::cluster}}]
Following Example \ref{eg:regression_adjust}, we first introduce some notation. 
Recalling that $\widehat{\bm{\gamma}}_z$ is the sample least squares coefficient of $\widetilde{Y}_j$ on $\widetilde{\bm{X}}_j$ for clusters in treatment arm $z$, we define $\widetilde{\bm{\gamma}}_z$ as the finite population least squares coefficient of $\widetilde{Y}_j(z)$ on $\widetilde{\bm{X}}_j$ for all clusters. Let
$s_z^2(\bm{\gamma}_z)$ be the sample variance of $\widetilde{Y}_j - \bm{\gamma}_z^\top
\widetilde{\bm{X}}_j$ in treatment arm $z$, and $\widehat{U}(\bm{\gamma}_1, \bm{\gamma}_0)=(M/N)^2\{s_1^2(\bm{\gamma}_1)/m_1+s_0^2(\bm{\gamma}_0)/m_0 \}$ be the variance estimator of $\widehat{\Delta}(\bm{\gamma}_1, \bm{\gamma}_0)$.
By the same logic as Example \ref{eg:regression_adjust}, under some regularity conditions,
first, $\widehat{\Delta}(\bm{\gamma}_1, \bm{\gamma}_0)$ is unbiased for $\tau$ and asymptotically Normal as $M\rightarrow\infty$ for any vectors $(\bm{\gamma}_1,\bm{\gamma}_0)$ that do not depend on the treatment indicators; second, $\widehat{\Delta}(\widetilde{\bm{\gamma}}_1, \widetilde{\bm{\gamma}}_0)$ is optimal in the sense of having the smallest sampling variance among all regression adjustment estimators defined in 
\eqref{eq:reg_adj_cluster}; third, the two regression adjustment estimators with true and estimated least squares coefficients, $\widehat{\Delta}(\widetilde{\bm{\gamma}}_1, \widetilde{\bm{\gamma}}_0)$ and $\widehat{\Delta}(\widehat{\bm{\gamma}}_1, \widehat{\bm{\gamma}}_0)$, have the same asymptotic Normal distribution, and the difference between 
them 
is of order $o_p(\sqrt{M}/N)$; fourth, the variance estimator $\widehat{U}(\bm{\gamma}_1, \bm{\gamma}_0)$ of $\widehat{\Delta}({\bm{\gamma}}_1, {\bm{\gamma}}_0)$ and confidence interval $|\widehat{\Delta}({\bm{\gamma}}_1, {\bm{\gamma}}_0)-\tau|\leq |\Phi^{-1}(\alpha/2)| \widehat{U}^{1/2}({\bm{\gamma}}_1, {\bm{\gamma}}_0)$ for $\tau$ are  asymptotically conservative, no matter whether $(\bm{\gamma}_1, \bm{\gamma}_0)$ are any vectors that do not depend on the treatment indicators or estimated from the samples, e.g., $(\widehat{\bm{\gamma}}_1, \widehat{\bm{\gamma}}_0)$. 
\end{proof}

\noindent
{\bf Example \ref{eg:factorial_design}.} 
We start to consider the joint asymptotic Normality of $\widehat{\tau}_k$'s under the sharp null hypothesis with fixed $Q=2^K$. 
Note that
\begin{align*}
\begin{pmatrix}
\widehat{\tau}_1\\
\vdots\\
\widehat{\tau}_{Q-1}
\end{pmatrix}
 = 2^{-(K-1)}\sum_{q=1}^Q 
\begin{pmatrix}
g_{1q}\\
\vdots\\
g_{Q-1,q}
\end{pmatrix}
\widehat{\bar{Y}}(q).
\end{align*}
We verify the two regularity conditions in Corollary \ref{cor:add_causal_effects} seperately. Under the sharp null hypothesis,
Condition \eqref{eq:condition_exp_add} is equivalent to $\max_{1\leq q\leq Q}m_N/(n_qv_N)\rightarrow0$ as $N\rightarrow\infty.$ The second regularity condition depends on the correlation between two factorial effect estimators. For any $1\leq k,m\leq Q-1$, according to Theorem \ref{thm::repeated-sampling}, the correlation between  $\widehat{\tau}_k$ and $\widehat{\tau}_m$ equals 
\begin{align*}
\text{Corr}_0(\widehat{\tau}_k, \widehat{\tau}_m) & = \frac{\Cov_0(\widehat{\tau}_k, \widehat{\tau}_m)}{
\sqrt{\Var_0(\widehat{\tau}_k)\Var_0(\widehat{\tau}_m)} 
}
= \frac{2^{-2(K-1)}\sum_{q=1}^{Q} n_q^{-1}g_{kq}g_{mq} v_N}
{2^{-2(K-1)}
\sum_{q=1}^Q n_q^{-1}v_N}
= \frac{\sum_{q=1}^{Q} n_q^{-1}g_{kq}g_{mq}}
{
\sum_{q=1}^Q n_q^{-1}}, 
\end{align*}
which has a limiting value if as $N\rightarrow\infty$, $n_q/N$ has a positive limiting value for all $q$. 

We then allow the total number of treatment combinations $Q$, or equivalently the number of factors $K$, to increase as the sample size increases. We consider the marginal asymptotic distribution of the $k$th average factorial effect estimator under the sharp null hypothesis.

\begin{theorem}\label{thm:CLT_null_fac_increase_K}
For $2^K$ factorial experiments, under the sharp null hypothesis, as $N\rightarrow \infty$ with possibly increasing $Q$ or $K$,  
if 
\begin{align}\label{eq:CLT_null_fac_increase_K}
\frac{
\max_{1\leq q\leq Q} ( n_q^{-2} )
}{
\sum_{q=1}^Q n_q^{-1} 
} \frac{m_N}{v_N} \rightarrow 0
\end{align}
then $\widehat{{\tau}}_k/\Var_0^{1/2}( \widehat{\tau}_{k} )\converged \mathcal{N}(0,1).$
\end{theorem}

Before presenting the proof of Theorem \ref{thm:CLT_null_fac_increase_K}, we discuss its implications. 
First, we consider the case where the number of factors $K$, or equivalently the number of treatment combinations $Q$ is fixed. Note that 
\begin{align*}
\frac{
\max_{1\leq q\leq Q} ( n_q^{-2} )
}{
\sum_{q=1}^Q n_q^{-1} 
} \frac{m_N}{v_N} & = 
\frac{
\max_{1\leq q\leq Q} ( n_q^{-1} )
}{
\sum_{q=1}^Q n_q^{-1} 
} \cdot \max_{1\leq q\leq Q} (n_q^{-1}) \cdot\frac{m_N}{v_N}
\begin{cases}
\leq \max_{1\leq q\leq Q} \frac{m_N}{n_qv_N}, \\
\geq Q^{-1}\max_{1\leq q\leq Q} \frac{m_N}{n_qv_N}.
\end{cases}
\end{align*}
Therefore, Condition \eqref{eq:CLT_null_fac_increase_K} reduces to $\max_{1\leq q\leq Q}m_N/(n_qv_N)\rightarrow0$ as $N\rightarrow\infty,$ which is equivalently to Condition \eqref{eq:condition_exp_add} in Corollary \ref{cor:add_causal_effects} in this case.

Second, we consider the case where $K$, or equivalently $Q$, increases with sample size $N$. 
When the design is balanced, i.e., $n_1=\ldots=n_Q=N/Q$, the quantity in Condition \eqref{eq:CLT_null_fac_increase_K} reduces to $m_N/(Nv_N)$, and therefore in this case Condition \eqref{eq:CLT_null_fac_increase_K} is equivalent to Condition \eqref{eq:condi_CLThajek} in Theorem \ref{thm:CLThajek}. 
For general unbalanced designs, 
note that 
\begin{align*}
\frac{
\max_{1\leq q\leq Q} ( n_q^{-2} )
}{
\sum_{q=1}^Q n_q^{-1} 
} \frac{m_N}{v_N}
\leq 
\frac{
\max_{1\leq q\leq Q} ( n_q^{-2} )
}{
Q \min_{1\leq q\leq Q} ( n_q^{-1} )
} \frac{m_N}{v_N}
 = 
\frac{1}{Q}
\frac{\max_{1\leq q\leq Q}n_q}{\min_{1\leq q\leq Q} (n_q^2)}\frac{m_N}{v_N}.
\end{align*}
If there exist constants $\underline{n}$ and $\overline{n}$ such that $0<\underline{n}\leq n_q\leq \overline{n} <\infty$ for any $q$ and $N$, a sufficient condition for the asymptotic Normality of $\widehat{{\tau}}_k$ becomes
${m_N}/(Qv_N) = 2^{-K} m_N/v_N \rightarrow 0.$

\begin{proof}[{\bf Proof of Theorem \ref{thm:CLT_null_fac_increase_K}}]
We use Lemma \ref{lemma:fraser} to establish the asymptotic Normality of $\widehat{\tau}_k$. First, we construct a finite population of size $N^2$ as follows: for any $1\leq i,j\leq N,$
\begin{align*}
C_{N}(i,j) =
\begin{cases}
n_1^{-1}g_{k1}Y_i, & 1\leq j \leq n_1, \\
n_2^{-1}g_{k2}{Y}_i, & 
n_1+1 \leq j\leq n_1+n_2, \\
\quad \quad \vdots & \\
n_Q^{-1}g_{kQ}{Y}_i, & 
\sum_{r=1}^{Q-1}n_r+1 \leq j\leq N.
\end{cases} 
\end{align*}
Let $(J_1,\ldots,J_N)$ be a random vector has probability $(N!)^{-1}$ to be any permutation of $\{1,\ldots,N\}$, and $L_i=q$ if  $\sum_{r=1}^{q-1}n_k < J_i \leq \sum_{r=1}^{q}n_k$. Then $(L_1,\ldots, L_N)$ is a random partition of the $N$ units into $Q$ groups of size $(n_1,\ldots,n_Q)$, and
\begin{align*}
2^{(K-1)}\widehat{{\tau}}_k & =   \sum_{q=1}^Q g_{kq}\widehat{\bar{Y}}(q)
=\sum_{q=1}^Q
\sum_{i=1}^N 1\{L_i=q\} n_q^{-1}g_{kq}Y_i =  \sum_{i=1}^N  n_{L_i}^{-1}g_{kL_i}{Y}_i 
=  \sum_{i=1}^N  C_{N}(i,J_i).
\end{align*}
According to Lemma \ref{lemma:fraser}, 
a sufficient condition for the asymptotic Normality of $\widehat{{\tau}}_k$ is Condition (b) in Lemma \ref{lemma:fraser}. Below we simplify Condition (b).
Note that 
\begin{align*}
D_{N}(i,j) & = C_{N}(i,j) - \frac{1}{N}\sum_{i'=1}^N C_{N}(i',j) - \frac{1}{N}\sum_{j'=1}^N C_{N}(i,j') + \frac{1}{N^2}\sum_{i'=1}^N \sum_{j'=1}^NC_{N}(i',j')\\
& =
\begin{cases}
n_1^{-1}g_{k1} (Y_i-\bar{Y}),
& 1\leq j \leq n_1, \\
n_2^{-1}g_{k2} (Y_i-\bar{Y}), & 
n_1+1 \leq j \leq n_1+n_2, \\
\quad \quad \quad \vdots & \\
n_Q^{-1}g_{kQ} (Y_i-\bar{Y}), & 
\sum_{r=1}^{Q-1}n_r+1 \leq j \leq N.
\end{cases}
\end{align*}
Thus, the quantity in Condition (b) of Lemma \ref{lemma:fraser} satisfies
\begin{align}\label{eq:factorial_sharp_null}
\frac{
\max_{1\leq i,j\leq N}D_{Nq}^2(i,j)
}{
N^{-1}\sum_{i=1}^N\sum_{j=1}^N D_{Nq}^2(i,j)
} & = \frac{
\max_{1\leq i\leq N}\max_{1\leq q\leq Q} \{ n_q^{-2}g_{kq}^2(Y_i-\bar{Y})^2 \}
}{
N^{-1}\sum_{i=1}^N \sum_{q=1}^Q n_q n_q^{-2}g_{kq}^2(Y_i-\bar{Y})^2
} = 
\frac{
\max_{1\leq q\leq Q} ( n_q^{-2} ) \cdot \max_{1\leq i\leq N}(Y_i-\bar{Y})^2
}{
N^{-1}\sum_{q=1}^Q n_q^{-1} \cdot \sum_{i=1}^N (Y_i-\bar{Y})^2
}
\nonumber
\\
& = \frac{N}{N-1}
\frac{
\max_{1\leq q\leq Q} ( n_q^{-2} )
}{
\sum_{q=1}^Q n_q^{-1} 
} \frac{m_N}{v_N}. 
\end{align}
Therefore, if Condition \eqref{eq:CLT_null_fac_increase_K} holds, then $\widehat{{\tau}}_k/\Var_0^{1/2}( \widehat{\tau}_{k} )\converged \mathcal{N}(0,1).$ 

\end{proof}

\end{document}